\documentclass[12pt]{article}
\RequirePackage[OT1]{fontenc}
\RequirePackage{amsthm,amsmath}
\usepackage{amssymb}
\usepackage{amsfonts}
\usepackage{mathrsfs, color}
\usepackage{url}
\usepackage{graphicx}

\RequirePackage[numbers]{natbib}
\RequirePackage[colorlinks,citecolor=blue,urlcolor=blue]{hyperref}

\theoremstyle{plain}
\newtheorem{thm}{Theorem}[section]
\theoremstyle{lemma}
\newtheorem{lem}[thm]{Lemma}
\newtheorem{rmk}[thm]{Remark}
\newtheorem{prop}[thm]{Proposition}
\newtheorem{cor}[thm]{Corollary}
\theoremstyle{definition}

\newtheorem{condition}{Condition}
\def\vX{ X}

\def\vx{\mathbf x}

\def\vg{\mathbf g}
\def\vY{Y}
\def\vW{W}
\def\vZ{Z}
\def\X{\mathbf X}
\def\cI{\mathcal I}
\def\cB{\mathcal B}

\def\vzero{0}

\def\vomg{{\boldsymbol{\omega}}}
\def\vxi{{\boldsymbol{\xi}}}
\def\vzeta{{\boldsymbol{\zeta}}}

\def\vXbar{\bar{X}}
\def\vYbar{\bar{Y}}
\def\P{\mathbb{P}}

\def\E{\mathbb{E}}
\def\Cov{\mathrm{cov}}
\def\Var{\mathrm{Var}}
\def\sumn{\sum_{i=1}^n}
\def\sumj{\sum_{j=1}^p}
\def\sumjk{\sum_{j,k=1}^p}
\newcommand{\sumne}[2]{\sum_{#1\ne #2}}
\def\tr{\mathrm{tr}}
\def\cum{\mathrm{cum}}
\def\I{\mathrm I}
\def\II{\mathrm{II}}
\def\III{\mathrm{III}}
\def\IV{\mathrm{IV}}
\def\V{\mathrm{V}}
\def\VI{\mathrm{VI}}
\def\VII{\mathrm{VII}}
\newcommand{\ind}[1]{\mathbb{I}\left\{#1\right\}}
\newcommand{\rbr}[1]{\left(#1\right)}
\newcommand{\rbR}[1]{\left[#1\right]}
\newcommand{\rBr}[1]{\left\{#1\right\}}
\newcommand{\rBR}[1]{\left|#1\right|}
\newcommand{\RBR}[1]{\left\Vert#1\right\Vert}
\newcommand{\cond}[1]{\left.#1\right|}
\newcommand{\rvs}[1]{{\rBR{#1}_2^2-f_1\over f}}

\def\sumstar{{\sum^\star}}

\def\g{g_{\psi,t}}

\numberwithin{equation}{section}

\newcounter{tictac}

\begin{document}

\begin{center}
{\large
    {\sc $L^2$ Asymptotics for High-Dimensional Data }
}
\bigskip

Mengyu Xu, Danna Zhang and Wei Biao Wu
\medskip
{\it

Department of Statistics\\
University of Chicago\\
5734 S. University Avenue\\
Chicago, Illinois 60637\\
USA\\
}
\end{center}
\centerline{November 24, 2014}

{\renewcommand\abstractname{Abstract}
\begin{abstract}
\baselineskip=18pt We develop an asymptotic theory for $L^2$ norms of sample mean vectors of high-dimensional data. An invariance principle for the $L^2$ norms is derived under conditions that involve a delicate interplay between the dimension $p$, the sample size $n$ and the moment condition. Under proper normalization, central and non-central limit theorems are obtained. To facilitate the related statistical inference, we propose a plug-in calibration method and a re-sampling procedure to approximate the distributions of the $L^2$ norms. Our results are applied to multiple tests and inference of covariance matrix structures.
\end{abstract}
\thispagestyle{empty}
\baselineskip=18pt

\section{Introduction}
\label{sec:introduction}

Let $\vX, \vX_i, i \in \mathbb{Z},$ be independent and identically distributed (i.i.d.) $p$-dimensional random vectors with mean $\E \vX_i = \mathbf{\mu}$ and covariance matrix $\Cov(\vX_i)=\Sigma$. Given the sample $\vX_1, \ldots, \vX_n$, we can estimate the mean $\mathbf{\mu}$ by the sample mean $\vXbar_n = \sumn\vX_i / n$. The primary goal of the paper concerns the asymptotic distribution of $|\vXbar_n - \mathbf{\mu}|^2 = (\vXbar_n - \mathbf{\mu})^T (\vXbar_n - \mathbf{\mu})$. The latter problem has a range of important applications in statistics including multiple tests and inference of covariance structures. Unless otherwise specified, assume throughout the paper that $\mathbf{\mu} = 0$.

In the classical setting with fixed dimension $p$, due to the Central Limit Theorem, we have $\sqrt n \vXbar_n \Rightarrow N(0, \Sigma)$. Hence, letting $\vY \sim N(0, \Sigma)$, we have by Slutsky's Theorem that 
\begin{eqnarray}\label{eq:A281020p}
\sup_{u \in \mathbb{R}}  |\P(n \vXbar_n^T\vXbar_n \le u) - \P( \vY^T\vY \le u)| \to 0.
\end{eqnarray}
In this paper we shall discuss the validity of (\ref{eq:A281020p}) in situations in which $p$ can be unbounded. In modern problems, the dimension $p$ can be larger than the sample size $n$. In this case, the traditional methods may not work. For example, \citet{MR863546} showed that the CLT is generally no longer valid when $p$ is large such that $\sqrt{n} = o(p)$. For other contributions see \citet{MR1965117, MR2144310}. Thus different methods are needed to prove (\ref{eq:A281020p}). The latter problem in the high dimensional setting and the corresponding statistical inference issues are challenging and have attracted wide attention. For linear processes, by \citet{MR1399305}, one can prove that $n \vXbar_n^T \vXbar_n - \tr(\X_n \X^T_n)/n$, where $\X_n = (\vX_1, \ldots, \vX_n)$ is the data matrix, is asymptotically Gaussian, assuming that $p/n$ tends to a finite constant and the largest eigenvalue of $\Sigma$ is negligible relative to its Frobenius norm. The latter condition can be violated in cases such as factor models, as discussed in \citet{MR3049913}, who studied the asymptotic distribution of $\vZ^T\vZ - \tr(\Sigma)$ over different types of $\Sigma$ under $\vZ \sim N(0, \Sigma)$.

In this paper, we shall develop an asymptotic theory for $\vXbar_n^T \vXbar_n$ for a generally distributed $\vX$, without requiring normality or linearity assumption. In particular, we shall apply the normal comparison method of Stein type and show that $\vXbar_n^T \vXbar_n$ can be approximated by a mixture of independent $\chi^2$ distributions. The approximate distribution may or may not be asymptotically normal. Specifically, we shall establish the following equivalent form of (\ref{eq:A281020p}):
\begin{eqnarray}\label{eq:A291140a}
\sup_{u \in \mathbb{R}}  |\P(n \vXbar_n^T\vXbar_n \le u) - \P(n \vYbar_n^T\vYbar_n \le u)| \to 0,
\end{eqnarray}
where $\vY_i, i \in \mathbb{Z},$ are i.i.d. $N(0, \Sigma)$ random vectors and $\vYbar_n = \sumn\vY_i / n$. We can view (\ref{eq:A291140a}) as an {\it invariance principle in a general sense} since the distributions of functions of non-Gaussian random vectors can be approximated by those of Gaussian vectors with the same covariance structure. The invariance principle in the narrow sense refers to the Gaussian approximation of partial sum processes of non-Gaussian random variables; cf \citet{MR3178474}.

As an immediate application of (\ref{eq:A281020p}) or (\ref{eq:A291140a}), one can perform the multiple test for the hypothesis 
\begin{eqnarray}\label{eq:A290726p}
H_0: \mu = \mu_0
\end{eqnarray}
for some pre-specified vector $\mu_0$. Assume without loss of generality that $\mu_0 = 0$.  A classical approach is to use the Hotelling $T^2$ statistic
\begin{eqnarray}\label{eq:A290611p}
T_n =  n \vXbar_n^T \hat \Sigma_n^{-1} \vXbar_n, 
\end{eqnarray}
where $\hat \Sigma_n = (n-1)^{-1} \sum_{i=1}^n (X_i - \bar X_n)  (X_i - \bar X_n)^T$ is the sample covariance matrix. In the high dimensional setting with $p > n$, $\hat \Sigma_n$ is singular and then $T_n$ is not well-defined. \citet{MR1399305} pointed out that this test lacks power. There is a large literature accommodating the Hotelling $T^2$ type statistic into the high-dimensional situation; see for example, \citet{MR0112207, MR0112208, MR1399305, MR2604697, MR2993891}, among others. \citet{MR0112207, MR0112208, MR2993891} considered Gaussian vectors. For the non-Gaussian random vectors, existing works assume linear forms. Central limit theorems for quadratic forms of sample mean vectors were proved in \citet{MR1399305,  MR2604697,  katayamanew}.

We test the hypothesis $H_0$ by directly using the test statistic $n \vXbar_n^T\vXbar_n$. Given the significance level $\alpha \in (0, 1)$, let $u_{1-\alpha}$ be the $(1-\alpha)$th quantile of $\vY^T\vY$. Namely $\P( \vY^T\vY \le u_{1-\alpha}) = 1-\alpha$. Then $H_0$ is rejected if $n \vXbar_n^T\vXbar_n > u_{1-\alpha}$. By (\ref{eq:A281020p}), the latter test has an asymptotic level $\alpha$.

If $\Sigma$ is known, the cutoff value $u_{1-\alpha}$ can be easily computed, either numerically or analytically, since the distribution of $\vY^T\vY$ is completely known. In most applications, however, $\Sigma$ is not known. We consider two approaches. The first one is to use an estimate of $\Sigma$. With the estimated covariance matrix, we can simulate a cutoff value. To access the goodness of the cutoff value with estimated covariance matrices, we shall introduce a new matrix convergence criterion: the {\it normalized consistency}. It is closely related, but different from the widely used spectral norm convergence. From modern random matrix theory, it is now well-known that the sample covariance matrix $\Sigma_n$ is not a (spectral norm) consistent estimator of $\Sigma$ when $p$ is large; see \citet{marvcenko1967distribution, MR2567175, MR0467894, MR0566592, MR950344, MR1863961, MR2485012}, to name a few. However, our results indicate that the sample covariance matrix can be normalized consistent in spectral norm, and hence the corresponding estimated cutoff value is  consistent. The normalized consistency guarantees the validity of resampling procedures.  Details are given in Section \ref{sec:O230147p}. As our second approach, we use the subsampling technique, which avoids estimating $\Sigma$ or its eigenvalues; see Section \ref{sec:O230149p}.

Another type of approach for testing (\ref{eq:A290726p}) is to use the maximum or $L^\infty$ norm $|\vXbar_n|_\infty = \max_{j \le p} |{\bar X}_{n j}|$ or the studentized version $\max_{j \le p} |{\bar X}_{n j}|/\hat \sigma_j$, where $\hat \sigma_j^2$ are estimates for the marginal variances $\sigma_j^2 = {\rm var}(X_{i j})$. \citet{MR2351093} considered the uniform consistency problem, and \citet{MR2372536} performed the $L^\infty$ test via Bonferroni correction, thus completely ignoring dependencies between entries of $X_i$. In a recent work, \citet{MR3161448} derived a Gaussian approximation for $|\vXbar_n|_\infty$ in the high-dimensional setting. In comparison with the marginal testing procedures, the procedure in \citet{MR3161448} is dependence-adjusted. \citet{MR3059419} established a deep {C}arm\'er-type moderate deviation principle for Hotelling's $T^2$ statistic under mild moment condition. The $L^2$-based test can be more powerful if the alternative consists of many small but non-zero signals that are of similar magnitudes.

This paper is organized as follows. In Section \ref{sec:gar}, we present the Gaussian approximation result. Section \ref{sec:est} provides a plug-in calibration of the Gaussian analogue when $\Sigma$ is unknown. We introduce {\it normalized consistency}, a new matrix convergence criterion. A sub-sampling procedure is also introduced there. In Section \ref{sec:cov_inf} we apply our result to the mean inference problem for linear processes. Section \ref{sec:cov} deals with the covariance matrix structure inference for linear processes. Proofs are given in Sections \ref{sec:proof}.

We now introduce some notation. For a vector $\vx = (x_1, \ldots, x_m)^T$, let the length $|\vx| = |\vx|_2 = (\vx^T\vx)^{1/2}$. Here $\vx^T\vx = \sum_{i=1}^m x_i^2$. Let $X$ be a random vector. Write $X\in {\cal L}^q$, $q > 0$, if $\RBR{X}_q := (\E |X|^q)^{1/q} < \infty$. For a matrix $A = (a_{jk})_{j,k}$, $\rho(A) = \max_\vx | A \vx| / |\vx|$ (resp. $\rBR{A}_F= (\sum_{jk}a_{jk}^2)^{1/2}$) denotes its spectral (resp. Frobenius) norm.  Write the $p \times p$ identity matrix as ${\rm Id}_p$. Denote by $C$ a positive constant whose value may vary from place to place.

\section{Main Result}
\label{sec:gar}
Consider i.i.d. random vectors $\vX, \vX_i\in\mathbb{R}^p$, $i \in \mathbb{Z}$, with $\E \vX_i = \vzero$ and covariance matrix $\Cov(\vX_i) = \Sigma$. Let $\Sigma = Q \Lambda Q^T$ be its eigen-decomposition, where $Q$ is an orthonormal matrix with $Q^T Q = {\rm Id}_p$ and $\Lambda = \mathrm{diag} \rbr{\lambda_1, \ldots, \lambda_p}$, with $\lambda_1 \ge \ldots \ge \lambda_p \ge 0$. Given data $X_1, \ldots, X_n$, let $\hat\Sigma = n^{-1} \X_n \X_n^T$, where $\X_n = (\vX_1, \ldots, \vX_n)$, be the sample covariance matrix; let  $\hat \lambda_1 \ge \ldots \ge \hat \lambda_p \ge 0$ be the eigenvalues of $\hat\Sigma$. Define 
\begin{eqnarray*}
f_k := [\tr(\Sigma^k)]^{1/k} \mbox{ and } \hat f_k:=[\tr(\hat \Sigma^k)]^{1/k}, \quad k=1,2,\ldots.
\end{eqnarray*}
Then $f_k^k = \sum_{i=1}^p \lambda_i^k$ and $\hat f_k^k = \sum_{i=1}^p \hat \lambda_i^k$. For the Frobenius norm with $k = 2$, we simply write $f = f_2$ and $\hat f = \hat f_2$.

Our main result is Theorem \ref{thm:true_sigma} which asserts that under suitable conditions the distributions of quadratic functions of $\bar X_n^T \bar X_n$ and  $\bar Y_n^T \bar Y_n$ are asymptotically close. In our asymptotic relation, we let $n \to \infty$ and view the dimension $p = p_n$ which satisfies $p_n \to \infty$ as $n \to \infty$. To state the theorem, we need to impose the following condition on $\vX$.

\begin{condition}
Let $\delta>0$. Assume that
\begin{align}\label{cond_1}
K_\delta(X) ^{2+\delta}:= \E\left|\rvs{\vX_1}\right|^{2+\delta}& < \infty;\\
\label{cond_2}
D_\delta(X) ^{2+\delta}:= \E\left| {{\vX_1^T \vX_2} \over f}\right|^{2+\delta}& < \infty.
\end{align}
\end{condition}

In conditions (\ref{cond_1}) and (\ref{cond_2}), $K_\delta(X)$ and $D_\delta(X)$ depend on the distribution of $X$. In the sequel for notational convenience we abbreviate them as $K_\delta$ and $D_\delta$, respectively. Note that $D_{0}=1$. In Sections \ref{sec:cov_inf} and \ref{sec:cov} we shall bound $K_\delta$ and $D_\delta$ for mean and covariance matrix inference problems arising from linear processes. Remark \ref{rmk:O231000p} provides an upper bound for moments of sums of dependent random variables using Rosenblatt transforms. Proposition \ref{prop:cond_Y} shows that for Gaussian vectors we can have explicit upper bounds.

\begin{prop}\label{prop:cond_Y}
Let $\vY_i$ be i.i.d. $N(\mathbf{0},\Sigma)$ and $\delta \ge 0$. Then
\begin{align}
\label{cond_1Y}
\E\left|\rvs{\vY_1}\right|^{2+\delta}&\le c_\delta^{2+\delta};\\
\label{cond_2Y}
\E\left|{ \vY_1^T \vY_2 \over f}\right|^{2+\delta}&\le d_\delta^{2+\delta},
\end{align}
where $c_\delta=(1+\delta)^{1/2} \|\xi^2-1\|_{2+\delta}$, $d_\delta=(1+\delta)^{1/2} \|\xi\|^2_{2+\delta}$ and $\xi \sim N(0, 1)$.
 \end{prop}

Based on (\ref{cond_1}) and (\ref{cond_2}), we have the following asymptotic result. Let $\eta_i, i \in \mathbb{Z},$ be i.i.d. $\chi^2_1$ random variables. Consider the normalized version
\begin{eqnarray}
\label{eq:M180141p}
R_n = { {n |\bar X_n|_2^2 - f_1} \over f} . 
\end{eqnarray}

\begin{thm}\label{thm:true_sigma}
Assume that (\ref{cond_1}) and (\ref{cond_2}) hold with $0 < \delta \le 1$. Then 
\begin{align}\label{eq:A291103p}
\sup_t\left|\P\left( R_n \le t\right)
      -\P\left( V \le t \right)\right|= O(\psi_n^{-1/2}),
      \mbox{ where } V = \sum_{j=1}^p { {\lambda_j} \over f} (\eta_j-1).
\end{align}
Here $\psi_n$ is the solution to the equation $L_\delta(n, \psi) = \psi^{-1/2}$ with
\begin{eqnarray*}
 L_\delta(n,\psi) = \psi^2 ( {\tilde K_{0}^2 \over n} + {\tilde K_{0} \over n^{1/2}})
  + \psi^q [   {\tilde K_{\delta}^q \over n^{q-1}}
  +  { \E (X_1^T \Sigma X_1)^{q/2} \over { n^{\delta/2} f^q} }
  + {\tilde D_\delta^q \over n^{\delta}}],
\end{eqnarray*}
where $q = 2+\delta$,  $\tilde K_{\delta} = K_{\delta} + c_{\delta}$, $\tilde D_{\delta} = D_{\delta} + d_{\delta}$, and $c_\delta$ and $d_\delta$ are given in Proposition \ref{prop:cond_Y}.  In particular, we have $\psi_n \to \infty$ if
\begin{eqnarray}\label{eq:O230925p}
{\tilde K_{0}^2 \over n}  
  +    {\tilde K_{\delta}^q \over n^{q-1}}
  +  { \E (X_1^T \Sigma X_1)^{q/2} \over { n^{\delta/2} f^q} }
  + {\tilde D_\delta^q \over n^{\delta}} \to 0 
   \mbox{ as }  n \to \infty.
\end{eqnarray}
Consequently the left hand side of (\ref{eq:A291103p}) converges to $0$.
\end{thm}

Note that $\E(\eta_i-1)^2 = 2$. By Lindeberg's Central Limit Theorem, 
\begin{eqnarray*}
V =  \sum_{j=1}^p f^{-1} \lambda_j (\eta_j-1)  \Rightarrow N(0, 2)
\end{eqnarray*}
holds if and only if $\lambda_1 / f = \rho(\Sigma) / f \to 0$. In this case by Theorem \ref{thm:true_sigma}, $R_n$ is also asymptotically $N(0, 2)$. In the previous literature, the primary focus is on the asymptotic normality of $\bar X_n^T \bar X_n$ or its modified version; see for example \citet{MR1399305, MR2483435, MR2604697}. As an exception, \citet{MR3049913} considered situations in which the CLT fails. If $\lambda_1 / f$ does not converge to $0$, $R_n$ may not have a Gaussian limit. When the dependence between entries of $\vX$ is strong, the asymptotic distribution of $R_{n}$ can be non-normal. For example, suppose $Y \sim N(0, \Sigma)$ and $\Sigma$ is Toeplitz with diagonal $1$ and $\sigma_{j,k}\sim |k-j|^{-D}$ for some $0<D<1/2$ as $|k-j| \to \infty$. Then $(Y^T Y - f_1)/f \Rightarrow \sum_{j=1}^\infty c_j (\eta_j - 1)$, the Rosenblatt distribution, with $c_j \sim c j^{D-1}$ as $j \to \infty$, and $c$ is a constant; see \citet{MR3079303}.

\begin{rmk}
{\rm
Since $X_1^T \Sigma X_1 = \E( X_1^T X_2 X_2^T X_1 | X_1)$, by Jensen's inequality, 
\begin{eqnarray}
\label{eq:O311227p}
\E (X_1^T \Sigma X_1)^{q/2} \le \E( |X_1^T X_2 X_2^T X_1|^{q/2}) 
 =  \E( |X_1^T X_2|^{q}) = D_\delta^q f^q. 
\end{eqnarray}
So (\ref{eq:O230925p}) follows from $\tilde K_{0}^2 / n + {\tilde K_{\delta}^q / n^{q-1}} + {\tilde D_\delta^q / n^{\delta/2}} \to 0 $ as $ n \to \infty$. Namely if $n$ is sufficiently large such that $\tilde K_0^2 +  \tilde K_{\delta}^{q / (q-1)} + \tilde D_\delta^{2 q / \delta} = o(n)$, then the left hand side of (\ref{eq:A291103p}) holds with rate $\psi_n^{-1/2} \to 0$. \qed
}
\end{rmk}

\begin{rmk}
{\rm
If Conditions (\ref{cond_1}) and (\ref{cond_2}) hold with $K_\delta$ and $D_\delta$ bounded, then we can choose $\psi_n \asymp n^{\delta/(5+2\delta)}$ and  the corresponding convergence rate in (\ref{eq:A291103p}) is  $O(n^{-\delta/(10+4\delta)})$.
\qed
}
\end{rmk}

\begin{rmk}
\label{rmk:O231000p}
{\rm Using the Rosenblatt transform (\cite{MR0049525}), we can find measurable functions $G_1, \ldots, G_p$ and i.i.d. standard uniform random variables $U_1, \ldots, U_p$ such that $X_1$ and the random vector $(G_1({\cal U}_1), \ldots, G_p({\cal U}_p))^T$ are identically distributed. Here ${\cal U}_j = (U_1, \ldots, U_j)$. Following Wu \cite{MR2172215}, define the predictive dependence measure $\theta_{i,j, q} = \| {\cal P}_i G^2_j({\cal U}_j) \|_q$, where $ {\cal P}_i \cdot = \E(\cdot | {\cal U}_i) -  \E(\cdot | {\cal U}_{i-1})$ is the projection operator. Since $X_1^T X_1 - f_1 = \sum_{i=1}^p {\cal P}_i X_1^T X_1 =  \sum_{i=1}^p \sum_{j=i}^p {\cal P}_i G^2_j({\cal U}_j)$, we have by Burkholder's inequality (p. 396 in \cite{MR1476912}) that  
\begin{eqnarray*}
 {{\|X_1^T X_1 - f_1 \|_q^2} \over {q-1}} \le \sum_{i=1}^p \|{\cal P}_i X_1^T X_1\|_q^2
  \le \sum_{i=1}^p \left( \sum_{j=i}^p \theta_{i,j, q} \right)^2.
\end{eqnarray*}
A similar upper bound also holds for the $L^q$ norm $\|X_1^T X_2 \|_q$.
\qed
}
\end{rmk}

To estimate the quantity $|\mu|_2^2 = \mu^T \mu$ based on i.i.d. vectors $X_1, \ldots, X_n$ with $\E X_i = \mu$, besides the natural plug-in estimator $\vXbar_n^T \vXbar_n$, we can also use the unbiased estimator $(n (n-1))^{-1} \sum_{i \not= j \le n} \vX_{i}^{T}\vX_{j}$; see also \citet{MR2604697}. This leads to the following variant of  (\ref{eq:M180141p}):
\begin{align}\label{eq:M220648p}
\tilde R_{n}=\frac{\sum_{i\ne j \le n}\vX_{i}^{T}\vX_{j}}{(n-1)f}
\end{align}
Using the arguments in the proof of Theorem \ref{thm:true_sigma}, without essential extra difficulties, we have the Gaussian approximation result:

\begin{cor} 
\label{cor:O311239p}
Assume Condition (\ref{cond_2}) and $\mu = 0$. Further assume  
\begin{eqnarray}
\label{eq:O231017p}
 L^\dagger_\delta :=  { \E (X_1^T \Sigma X_1)^{q/2} \over { n^{\delta/2} f^q} }
  + {\tilde D_\delta^q \over n^{\delta}} \to 0.
\end{eqnarray}
Then $\psi_n := (L^\dagger_\delta)^{-1/(q+1/2)} \to \infty$ and, recall $V =  \sum_{j=1}^p f^{-1} \lambda_j (\eta_j-1)$,
\begin{align}\label{eq:M220656p}
\sup_t |\P(\tilde R_n \le t)  - \P\left(V \le t \right) |
 = O(\psi_n^{-1/2}) \to 0.
\end{align}
\end{cor}
By (\ref{eq:O311227p}), a simple sufficient condition for (\ref{eq:O231017p}) is $D_\delta^q = o(n^{\delta / 2})$. Then the rate in (\ref{eq:M220656p}) becomes $D_{\delta}^{q/(5+2\delta)} n^{-\delta/(10+4\delta)}$. Notice that in Corollary \ref{cor:O311239p} Condition (\ref{cond_1}) is not needed since $\tilde R_{n}$ does not involve the diagonal terms $\vX_i^{T} \vX_i$. Consequently the weaker moment condition $X_i \in {\cal L}^{2+\delta}$ suffices. In comparison,  (\ref{cond_1}) necessarily requires the stronger moment condition $X_i \in {\cal L}^{4 + 2\delta}$.  For linear processes, applying the results in \citet{MR1399305}, one can have a CLT for $\tilde R_{n}$ by assuming the existence of $4$th moments, $p/n$ tends to a finite constant and $\rho(\Sigma) / f \to 0$. Since $\rho(\Sigma)^4 \le f_4^4 \le \rho(\Sigma)^2 f^2$, the latter condition is equivalent to $f_4^4/f^4 = o(1)$, which is also imposed in \cite{MR2604697, MR2724863}. In comparison, by (\ref{linear_cond_2}) of Theorem \ref{verify_cond}, it suffices to impose a weaker $(2+\delta)$th moment condition, and our result (\ref{eq:M220656p}) can allow non-Gaussian limiting distributions.

\begin{rmk}
{\rm 
In general the condition $L^\dagger_\delta \to 0$ in (\ref{eq:O231017p}) is not relaxable for the following result
\begin{eqnarray}
 \label{eq:O231021p}
\sup_t |\P(\tilde R_n \le t) - \P\left( V \le t \right) | \to 0.
\end{eqnarray}
Let $\ell = p^\beta$, $\beta > 1/2$, and let $B_{i j}, i, j \in \mathbb Z,$ be i.i.d. Bernoulli($\ell^{-1}$) random variables; let $X_{i j} = (\ell B_{i j}-1) (\ell-1)^{-1/2}$. Then $\E X_{i j} = 0$, $\E X_{i j}^2 = 1$, $\E |X_{i j}|^q \sim \ell^{q/2-1}$, $\Sigma = {\rm Id}_p$ and $f^2 = p$. By Burkholder's inequality, $\E |X_1^T X_1|^{q/2} \le c_q \E| \sum_{j=1}^p X_{1 j}|^q$. By Rosenthal's inequality (\cite{MR0271721}), $\E| \sum_{j=1}^p X_{1 j}|^q \le c_q 
( p \E |X_{11} |^q + p^{q/2})$ and $\E |X_1^T X_2|^q \le c_q( p \E |X_{11} X_{21}|^q + p^{q/2})$. Then (\ref{eq:O231017p}) requires that 
\begin{eqnarray}\label{eq:O241111p}
\ell = o(n p^{1/2}), \mbox{ or } p^{\beta - 1/2} = o(n).
\end{eqnarray}
We remark that Condition (\ref{eq:O241111p}) is also necessary for (\ref{eq:O231021p}). By (\ref{eq:O231021p}),
\begin{eqnarray}\label{eq:O241015p}
{ {(n-1) f \tilde R_n} \over { n p^{1/2} } }
 =
{ {\sum_{l=1}^p Q_l} \over { n p^{1/2} } }   \Rightarrow N(0, 2), \mbox{ where }
 Q_l = \sum_{i \not= j \le n} X_{i l} X_{j l}.
\end{eqnarray}
By the Linderberg-Feller central limit theorem, (\ref{eq:O241015p}) holds if and only if
\begin{eqnarray}\label{eq:O241018p}
p \E\{ [Q_1/(n p^{1/2})]^2 {\bf 1}_{|Q_1| \ge \theta n p^{1/2}}\}
 = \E\{  n^{-2} Q^2_1 {\bf 1}_{|Q_1| \ge \theta n p^{1/2}}\}
 \to 0 
\end{eqnarray}
holds for every $\theta > 0$. Note that $W := \sum_{i=1}^n B_{i 1}$ is binomial($n, \ell^{-1}$). If
\begin{eqnarray}\label{eq:O241026p}
n p^{1/2} = O(\ell),
\end{eqnarray}
then for all large $n$, the event $\{ |Q_1| < \theta n p^{1/2} \}$ implies $\{ W \le 1 \}$, and
\begin{eqnarray}\label{eq:O241030p}
\E\{  n^{-2} Q^2_1 {\bf 1}_{|Q_1| < \theta n p^{1/2}}\}
\le \E\{  n^{-2} Q^2_1 {\bf 1}_{W \le 1}\}\le { n^2 \over \ell^2 } + { n \over \ell } \to 0,
\end{eqnarray}
by noting that $\E\{  n^{-2} Q^2_1 {\bf 1}_{W = 0} \} \le n^2 \ell^{-2}$ and $\E\{  n^{-2} Q^2_1 {\bf 1}_{W = 1} \} \le n \ell^{-1}$.  Clearly (\ref{eq:O241030p}) violates (\ref{eq:O241018p}) since $n^{-2} \E Q^2_1 \to 2$. 
\qed
}
\end{rmk}

\begin{rmk}
{\rm
A careful check of the proof of Theorem \ref{thm:true_sigma} indicates that the result therein still holds for independent, but not identically distributed random vectors $X_i$ with mean $0$, (same) covariance matrix $\Sigma$: we need to replace the quantities $K_\delta$, $D_\delta$ and $\E (X_1^T \Sigma X_1)^{q/2}$ therein by $K_{\delta, n} := \max_{i \le n} \| X_i^T X_i - f_1\|_q / f$, $D_{\delta, n} := \max_{i < l \le n} \| X_i^T X_l \|_q / f$ and $\max_{i \le n} \E (X_i^T \Sigma X_i)^{q/2}$, respectively.
\qed
}
\end{rmk}

\section{Re-sampling Calibration Procedures}
\label{sec:est}
To test the hypothesis $H_0: \mu = 0$ (say) at level $\alpha \in (0, 1)$ using Theorem \ref{thm:true_sigma}, we need to compute the $(1-\alpha)$th quantile of the approximate distribution 
\begin{eqnarray}\label{eq:M180343p}
V = \sum_{j=1}^p f^{-1} \lambda_j (\eta_j-1). 
\end{eqnarray}
In practice, however, $\Sigma$ and hence $\lambda_j$ are not known. Section \ref{sec:O230147p} proposes an approach based on estimated $\lambda_j$. An alternative subsampling approach is given in Section \ref{sec:O230149p} which avoids estimating eigenvalues.

\subsection{A Plug-in Procedure and Normalized Consistency}
\label{sec:O230147p}

As a natural way to approximate the distribution of $V$, one can replace $\lambda_j$'s in (\ref{eq:M180343p}) by their estimates. Let $\tilde \Sigma$ be an estimate of $\Sigma$ based on the data $\X_{n} = (\vX_1, \ldots, \vX_n)$; let $\tilde \lambda_1 \ge \ldots \ge \tilde \lambda_p \ge 0$ be the eigenvalues of $\tilde \Sigma$ and $\tilde f = (\sum_{j=1}^p \tilde \lambda_j^2)^{1/2}$. Let $\tilde V = \sum_{j=1}^p \tilde f^{-1} \tilde \lambda_j (\tilde \eta_j-1)$, where $\tilde \eta_j$ are i.i.d. $\chi^2_1$ random variables that are independent of $\X_{n}$. By Lemma \ref{lem:com_mix_chisq}, if
\begin{eqnarray}
\label{eq:M01959p}
 \max_{j \le p}  | f^{-1} \lambda_j  - \tilde f^{-1} \tilde \lambda_j| \to 0 \mbox{ in probability},
\end{eqnarray}
then with probability converging to $1$, we have
\begin{eqnarray}
\label{eq:M011001p}
 \sup_t  |\P(V \le t) - \P^*(\tilde V \le t)| \to 0, 
\end{eqnarray}
where $\P^*$ is the conditional probability given ${\bf X}_{n}$. With (\ref{eq:M011001p}), the distribution of $V$ can be approximated by that of $\tilde V$ via extensive simulations.

\begin{lem}\label{lem:com_mix_chisq}
Let $a_{p,1} \ge a_{p,2} \ge \ldots \ge a_{p,p} \ge 0$ and $b_{p,1} \ge b_{p,2} \ge \ldots \ge b_{p,p} \ge 0$ be two sequences of real numbers satisfying $\sumj a_{p,j}^2=\sumj b_{p,j}^2=1$. Assume $\max_{j \le p} \rBR{a_{p,j}-b_{p,j}} \to 0$. Let $\eta_j$ be i.i.d. $\chi^2_1$ random variables and $\eta_j' = \eta_j -1$. Let $V_a = \sumj a_{p,j} \eta'_j$ and $V_b = \sumj b_{p,j} \eta'_j$. Then  
\begin{align}\label{com_mix_chisq}
\sup_x\big|{\P\left( V_a \le x \right)-\P\left(V_b \le x \right)}\big| = o(1).
\end{align}
\end{lem}

Interestingly, there is a simple sufficient condition for (\ref{eq:M01959p}). By Weyl's theorem (\citet[Theorem 8.1.5]{MR3024913}), (\ref{eq:M01959p}) follows from
\begin{align}\label{normalized_const}
\rho({\tilde \Sigma / \tilde f - \Sigma/f}) = o_\P(1).
\end{align}
We say that an estimate $\tilde \Sigma$ of $\Sigma$ is {\it normalized consistent} if (\ref{normalized_const}) holds. It is closely related to, but quite different from the classical definition of spectral norm consistency in the sense of
\begin{align}\label{eq:M5011014p}
\rho({\tilde \Sigma - \Sigma}) = o_\P(1).
\end{align}
Normalized consistency does not generally imply the spectral norm consistency (\ref{eq:M5011014p}). For example, let $n = p$ and $X_i$ be i.i.d. standard $N(0, {\rm Id}_p)$ random vectors. By the random matrix theory, (\ref{eq:M5011014p}) does not hold for the sample covariance matrix $\hat \Sigma = n^{-1} \sum_{i=1}^n X_i X_i ^T$, which is not a consistent estimate of $\Sigma = {\rm Id}_p$; see \citet{marvcenko1967distribution, MR0467894, MR0566592}. Indeed, the largest eigenvalue of $\tilde \Sigma$ converges to $4$, while the smallest one converges to $0$. However the normalized consistency (\ref{normalized_const}) holds since both $\rho(\Sigma/f) = p^{-1/2} \to 0$ and $\rho(\hat \Sigma / \hat f) = O_\P(p^{-1/2}) \to 0$. Without further conditions, the spectral norm consistency (\ref{eq:M5011014p}) does not imply the normalized consistency either. Proposition \ref{prop:M020905a} relates these two types of convergence.

\begin{prop}
\label{prop:M020905a}
For an estimate $\tilde \Sigma$ of $\Sigma$ with $\tilde f = (\tr(\tilde \Sigma^2))^{1/2}$, assume that $\tilde f / f \to 1$ in probability. Then the normalized consistency (\ref{normalized_const}) holds if and only if $\rho({\tilde \Sigma - \Sigma}) = o_\P(f)$.
\end{prop}

Let $\tilde \Sigma$ be a normalized consistent estimate of $\Sigma$.  Given $\alpha \in (0, 1)$, let $\tilde v_{1-\alpha}$ be such that the conditional probability $\P^*(\tilde V \le \tilde v_{1-\alpha}) = 1-\alpha$; cf (\ref{eq:M011001p}). Then at level $\alpha$ we reject the null hypothesis $H_0: \mu = 0$ if the test statistic $\hat R_n :=  (n |\bar X_n|_2^2 - \hat f_1) / f^\dagger$ satisfies $\hat R_n > \tilde v_{1-\alpha}$, where $f^\dagger$ is a ratio consistent estimate of $f$, namely $f^\dagger / f - 1 = o_\P(1)$; see \citet{MR1399305, MR2604697}, and $\hat f_1 = (n-1)^{-1} \sum_{i=1}^n (X_i - \bar X_n)^T  (X_i - \bar X_n)$ is an unbiased estimate of $f_1$. Note that, interesting, the numerators of $\hat R_n$ and $\tilde R_{n}$ in (\ref{eq:M220648p}) are equivalent in view of $n |\bar X_n|_2^2 - \hat f_1 = (n-1)^{-1} \sum_{i \not= j \le n} X_i^T X_j$. It is easily seen that, if $\mu$ satisfies $n \mu^T \mu / f \to \infty$, then $H_0: \mu = 0$ is rejected with probability going to $1$.

Under certain structural assumptions such as bandedness and sparsity, various regularized procedures have been proposed so that the spectral norm consistency (\ref{eq:M5011014p}) holds; see \citet{MR2024760, MR2485008, MR2387969, MR2847973} among others. In our setting we do not make such structural assumptions, and therefore simply use the sample covariance matrix $\hat \Sigma$. Its normalized consistency is dealt with in Theorem \ref{thm:spec_Sigma_over_f}. It is interesting to study whether other covariance matrix estimates are normalized consistent.

\begin{thm}
\label{thm:spec_Sigma_over_f}
(i) Assume $\E[(\vX_1^T\vX_1)^2]=o\rbr{ nf^2}$.  Then
\begin{align}\label{spec_Sigma_over_f}
\E| \hat\Sigma / \hat f - \Sigma / f|^2_F = o(1),
\end{align}
which further implies the normalized consistency (\ref{normalized_const}). (ii) Assume $n f^2 = o\rBr{\E[(\vX_1^T\vX_1)^2]}$, (\ref{cond_1}) holds with $K_2=O(n^{3/4})$, and
\begin{align}
\label{high_moment2}
\E[(\vX_1^T\vX_2)^4]&=o\rBr{ \E^2[(\vX_1^T\vX_1)^2]}.
\end{align}
Then $\rho(\hat\Sigma / \hat f) = o_\P(1)$, and (\ref{normalized_const}) holds if and only if $\rho(\Sigma) = o(f)$.
\end{thm}

Theorem \ref{thm:spec_Sigma_over_f}(i) requires that $n$ is big enough such that $\E[(\vX_1^T\vX_1)^2] / f^2=o(n)$, and the approximate distribution $V$ in (\ref{eq:M180343p}) may or may not be asymptotically normal. The latter condition trivially holds if the entries of $X_1$ are strongly dependent in the sense that $f^2 = \sum_{j, k \le p} \sigma^2_{j, k} \asymp p^2$ and $\max_{j \le p} \|X_{1 j} \|_4 \le C$ for some constant $C$. In this case $\E[(\vX_1^T\vX_1)^2]  \le p^2 C^4$ and the condition $\E[(\vX_1^T\vX_1)^2] / f^2=o(n)$ reduces to the natural one $n \to \infty$. As a simple example, let $X_{1 j} = a_j Z + \xi_j$, where $Z, \xi_1, \ldots, \xi_p$ are i.i.d. $N(0, 1)$ and $a_j$ are real coefficients. If $\sum_{j=1}^p a_j^2 \asymp p$, then $\E[(\vX_1^T\vX_1)^2] / f^2 \asymp 1$ and the condition $n \to \infty$ suffices. In this case $\Sigma$ has $p-1$ eigenvalues $1$ and $1$ eigenvalue $1+\sum_{j=1}^p a_j^2$, hence $V \Rightarrow \chi_1^2 - 1$. Under Case (ii) with smaller $n$, however, normalized consistency of $\hat\Sigma$ necessarily requires that $\rho(\Sigma)=o(f)$.

Proposition \ref{prop:M181017} provides an expression for the quantity $\E[(\vX_1^T\vX_2)^4]$ in (\ref{high_moment2}). Its proof is routine and the details are omitted.

\begin{prop}\label{prop:M181017}
We have the cumulants expression
\begin{eqnarray*}
\E[(\vX_1^T\vX_2)^4] 
 &=& 3f^4+6f_4^4+6\sum_{1\le j,k,m,q\le p}\cum(X_{1j},X_{1k},X_{1m},X_{1q})\sigma_{km}\sigma_{qj}\\
 &&+\sum_{1\le j,k,m,q\le p}\cum(X_{1j},X_{1k},X_{1m},X_{1q})^2.
\end{eqnarray*}
\end{prop}

\subsection{A Subsampling Procedure}
\label{sec:O230149p}
Let $m = m_n \in \mathbb N$ be such that $m \to \infty$ and $m = o(n)$; let the index set $B_j = \{ l \in \mathbb Z:  (j-1)m < l \le j m\}$, $j = 1, \ldots, L$, where $L = \lfloor n/m \rfloor$ and $\lfloor u \rfloor = \max\{k \in \mathbb Z: k \le u\}$. For a set $B \subset \{1, \ldots, n \}$, let $|B|$ be its cardinality. Define the empirical subsampling distribution function
\begin{eqnarray}
\label{eq:O240903a}
\hat F(t) = {1\over L} \sum_{j=1}^L {\bf 1}_{ m|\bar X_{B_j} - \bar X|_2^2 \le t(1-m/n)}, 
 \mbox{ where }
  \bar X_B = {{\sum_{b \in B} X_b} \over {|B|}}.
\end{eqnarray}
As a slightly different version, let $A_1, \ldots, A_J$ be i.i.d. uniformly sampled from the class $\mathcal{A} := \rBr{A: A\subset \{1, \ldots, n\}, \rBR{A}=m}$. Assume that the sampling process $(A_j)_{j \ge 1}$ and $(X_i)_{i \ge 1}$ are independent. Define
\begin{eqnarray}
\label{eq:O250918p}
\check F(t) = {1\over J} \sum_{j=1}^J {\bf 1}_{ m|\bar X_{A_j} - \bar X|_2^2 \le t(1-m/n)}.
\end{eqnarray}

\begin{thm}
\label{th:O240914a}
Let $0 < \delta \le 1$. Assume (\ref{cond_1}), (\ref{cond_2}), $m \to \infty$, $m = o(n)$, and (\ref{eq:O230925p}) holds with $n$ therein replaced by $m$. Then (i) 
\begin{eqnarray}
\label{eq:O240913a}
\sup_t |\hat F(t) - \P( n |\bar X - \mu|_2^2 \le t)| \to 0 \mbox{ in probability}. 
\end{eqnarray}
(ii) If $J \to \infty$, then the convergence (\ref{eq:O240913a}) also holds for $\check F(t)$. 
\end{thm}

Theorem \ref{th:O240914a} suggests that samples quantiles of $\hat F(\cdot)$ or $\check F(\cdot)$ can be used to approximate those of $F(t) = \P( n |\bar X - \mu|_2^2 \le t)$.  Given a level $\alpha \in (0, 1)$, let $\check v_{1-\alpha}$ be the $(1-\alpha)$th quantile of $\check F(\cdot)$. Then at level $\alpha$ we can reject the null hypothesis $H_0: \mu = 0$ if $n |\bar X|_2^2 \ge \check v_{1-\alpha}$. Similarly as the plug-in approach, if $n \mu^T \mu / f \to \infty$, then $H_0$ is rejected with probability going to $1$.

{\it Proof of Theorem \ref{th:O240914a}.} (i) Assume without loss of generality that $\mu = 0$. For a set $B \subset \{1, \ldots, n \}$ define $W_B^\circ = (|B| |\bar X_{B}|_2^2 - f_1 ) / f$ and $W_B = [|B| |\bar X_{B} - \bar X|_2^2 / (1- |B|/n) - f_1 ] / f$. Using the identity $n \bar X = |B| \bar X_{B} + (n-|B|) \bar X_{B^c}$, where $B^c = \{1, \ldots, n\} - B$, we have by elementary manipulations that
\begin{eqnarray}
W_B = { {n-|B|} \over n} W_B^\circ +  {{|B|} \over n} W_{B^c}^\circ
 -  {2 |B|}  { {n- |B| } \over n} {{ \bar X^T_{B} \bar X_{B^c} } \over f}.
\end{eqnarray}
Then for any $\theta > 0$, we have by the triangle inequality that 
\begin{align}\label{eq:O240954}
\P( W_{B_j}^\circ \le {{ t - \theta} \over {1-m/n}})
 - \tau
\le
\P(W_{B_j} \le t) \le \P( W_{B_j}^\circ \le {{t + \theta}\over{1-m/n}})
 + \tau,
\end{align}
where $\tau = \P( |R_j| \ge \theta)$, $R_j = (m/n) W_{B_j^c}^\circ - 2 m (1-m/n)^{-1} f^{-1}  \bar X^T_{B_j} \bar X_{B_j^c}$. Note that $\E|\bar X^T_{B_j} \bar X_{B_j^c}|^2 = f^2/(m(n-m))$. Since $m = o(n)$, by Theorem \ref{thm:true_sigma}, we have $R_j = o_\P(1)$ and $\tau \to 0$. Hence by Theorem \ref{thm:true_sigma}, Lemma \ref{lem:density} and (\ref{eq:O240954}), 
\begin{eqnarray}\label{eq:O241004a}
\P(W_{B_j} \le t) - \P( W_{B_j}^\circ \le t) \to 0. 
\end{eqnarray}
A similar argument implies that, for $j \not= j'$, the joint probability
\begin{eqnarray}\label{eq:O241009a}
\P(W_{B_j} \le t,   W_{B_{j'}} \le t) - \P( W_{B_j}^\circ \le t, W_{B_{j'}}^\circ \le t) \to 0. 
\end{eqnarray}
Therefore, by Theorem \ref{thm:true_sigma}, we have $\E| \hat F(t) - \P(V \le t)|^2 \to 0$, which implies the uniform version (\ref{eq:O240913a}) via the standard Glivenko--Cantelli argument in view of the continuity result Lemma \ref{lem:density}.

We now prove (ii). Following the argument in (i), it suffices to show that
\begin{eqnarray}\label{eq:O310956p}
\E|\P(W^\circ_{A_j} \le t,   W^\circ_{A_{j'}} \le t) - \P^2( V \le t)| \to 0.
\end{eqnarray}
For sets $A, A' \in \mathcal{A}$, let $A \cap A' = D_1$, $A - D_1 = D_2$ and $A'  - D_1 = D_3$. Then  
\begin{eqnarray*}
W^\circ_A=(1-k/m)W^\circ_{D_2}+(k/m)W^\circ_{D_1}+2k(1-k/m)\frac{\bar X_{D_1}^T\bar X_{D_2}}{f},
\end{eqnarray*}
where $k = |D_1|$. A similar expression exists for $W^\circ_{A'}$. Choose a sequence $\rho_n \to 0$ with $m / n = o(\rho_n)$. If $k \le m \rho_n$, similarly as in part (i), we have $|\P(W^\circ_{A} \le t,   W^\circ_{A'} \le t)-\P(W^\circ_{D_2} \le t,   W^\circ_{D_3} \le t)| \to 0$ and $|\P(W^\circ_{D_2} \le t) - \P(V \le t)| \to 0$. Note that $\E|A_j \cap A_{j'}| \le m^2 / n$. Then 
$\P( |A_j \cap A_{j'}| \ge m \rho_n) \le m / (n \rho_n) \to 0$. Then (\ref{eq:O310956p}) follows by conditioning on $|A_j \cap A_{j'}| \le m \rho_n$. \qed

\section{Applications to Linear Processes}
\label{sec:cov_inf}
In this section we shall apply our main result to the linear process
\begin{align}\label{linear}
\vX_{i}=A\vxi_{i} = A (\xi_{i1}, \ldots, \xi_{i p})^T,
\end{align}
where $\xi_{i j}, i,j \in \mathbb{Z}$, are i.i.d. random variables with mean $0$ and variance $1$ and $A$ is a coefficient matrix.  The linear form (\ref{linear}) is natural and rich. Similar forms were also used in \cite{MR1399305,MR2604697}, among others. Proposition \ref{verify_cond} generalizes Proposition \ref{prop:cond_Y} and it concerns conditions (\ref{cond_1}) and (\ref{cond_2}) where $\vX_i$ is of form (\ref{linear}).

\begin{prop}\label{verify_cond}
Assume (\ref{linear}) and that $\RBR{\xi_1}_{4+2\delta} < \infty$ for some $\delta > 0$. Let $\bar D_\delta=(1+\delta)\RBR{\xi_{1}}_{2+\delta}^{2}$ and $\bar K_\delta = 2 \|\xi_1^2\|_{2+\delta}$. Then
\begin{align}\label{linear_cond_1}
\E\rBR{\vX^T_1\vX_1 - f_1\over f}^{2+\delta}&\le \bar K_\delta^{2+\delta},\\
\label{linear_cond_2}
\E\left|{ {\vX_1^T \vX_2 }\over f}\right|^{2+\delta}&\le \bar D_\delta^{2+\delta}.
\end{align}
\end{prop}
\begin{proof}[Proof of Proposition \ref{verify_cond}] Let $q = 2+\delta$. Since $\vxi$ has covariance matrix ${\rm Id}_p$, $\Sigma=AA^T$. Denote $\rbr{b_{jk}}_{j,k}=B=A^T A$. Write 
$\vxi_1 = (\xi_1, \ldots, \xi_p)^T$ and $\vxi_2 = (\zeta_1, \ldots, \zeta_p)^T$. By Burkholder's inequality, (\ref{linear_cond_2}) follows from
\begin{eqnarray}\label{eq:M21827}
\RBR{\vX_1^T\vX_2}_q^{2}
&=&\left\|{\sum_{j=1}^p \xi_j \sum_{k=1}^p b_{jk}\zeta_k}\right\|_q^{2}\cr
&\le& (q-1)  \sum_{j=1}^p \|\xi_j\|_q^2  \left\|  \sum_{k=1}^p b_{jk}\zeta_k\right\|_q^{2} \cr
&\le& (q-1)^2   \|\xi_1\|_q^2 \sum_{j=1}^p  \sum_{k=1}^p b_{jk}^2 \left\|  \zeta_k\right\|_q^{2} \cr
&=& (q-1)^{2}\RBR{\xi_{1}}_q^{4}f^{2}.
\end{eqnarray}
Since $\xi_j \sum_{k<j}  b_{jk}\xi_k$ are martingale differences, we similarly have
\begin{align}\label{eq:M210757p}
\RBR{\vX^T\vX-f_1}_q^{2}&\le2\big\|{\sumj  b_{jj}\rbr{\xi_j^2-1}}\big\|_q^{2}+2\big\|{\sumne{j}{k}  b_{jk}\xi_j\xi_k}\big\|_q^{2}\cr
&\le 2(q-1)\RBR{\xi_{1}^{2}-1}_q^{2}{\sumj b_{jj}^2}+8(q-1)^{2}\RBR{\xi_{1}}_q^{4}\sum_{k<j}b_{jk}^{2}\cr
&\le 4 q^2 \|\xi_1^2\|_q^2 f^2,
\end{align}
which implies (\ref{linear_cond_1}). 
\end{proof}


We remark that (\ref{linear_cond_2}) actually holds under the weaker moment condition $\xi_i \in {\cal L}^{2+\delta}$. Proposition \ref{verify_cond} implies that the Gaussian approximation (\ref{eq:A291103p}) of Theorem \ref{thm:true_sigma} holds with convergence rate $O(n^{-\delta/(10+4\delta)})$ for linear processes.

\section{Inference of Covariance Matrices}\label{sec:cov}

In this section we shall apply our results to test hypotheses on covariance matrices. The latter problem has been extensively studied in the literature. Earlier papers focus on lower-dimensional case; see \citet{MR1990662, MR0092296, MR0339405, MR0312619}. The traditional likelihood ratio test can fail in the high-dimensional setting (cf. \citet{MR2572444}). Under the assumption that $p/n$ is bounded, or $p = O(n)$,  \citet{MR2572444, MR2234197, MR2328427} considered test of identity, sphericity, and diagonal covariance matrices. Recently, \citet{MR2724863} proposed test statistics for sphericity and identity, and proved the normality with no condition on $p/n$, with $f_4=o(f)$. \citet{MR3015026} considered testing whether a covariance matrix is banded. \citet{MR3127858} applied the empirical likelihood ratio test. Other contributions can be found in \citet{MR3160557, MR3113808, MR2189467, MR2719881, MR2911842, MR1926169}. In many of those papers it is assumed that $X_1$ is Gaussian.

Given the data $X_1, \ldots, X_n$, which are i.i.d. with mean $0$ and covariance matrix $\Sigma$, we test the null hypothesis $H_0: \Sigma = \Sigma_0 = (\sigma_{0, jk})_{j, k \le p}$. Let $\hat \Sigma = \sum_{i=1}^n X_i X_i^T / n$ be the sample covariance matrix. \citet{MR3054550} considered the $L^\infty$ test statistic $\max_{j, k \le p} |\hat\sigma_{jk} - \sigma_{0, jk}|$. The latter test is not powerful if the alternative hypothesis consists of many small but non-zero covariances. Here we shall study the test statistic
\begin{align}\label{eq:M20818p}
T_n = \sumjk\rbr{\hat\sigma_{jk}-\sigma_{0, jk}}^2.
\end{align}
We reject $H_0$ if $T_n$ exceeds certain cutoff values. The problem of deriving asymptotic distribution of $T_n$ has been open. In many of earlier papers it is assumed that $\Sigma_0$ has special structures such as being diagonal or spheric and/or $X_i$ is Gaussian or has independent entries. Here we shall obtain an asymptotic theory for $T_n$ for linear processes of form (\ref{linear}).

We shall apply Theorem \ref{thm:true_sigma}. For $u = (u_1, \ldots, u_p)^T$, let
\begin{align}\label{eq:M20813p}
\vW(u) = \rbr{\begin{array}{c}
u_1^2-\sigma_{11}\\
u_1 u_2-\sigma_{12}\\
\ldots\\
u_1 u_p-\sigma_{1p}\\
u_2 u_1-\sigma_{12}\\
\ldots\\
u_p^2-\sigma_{pp}
\end{array}}
\end{align}
be a $p^2$-dimensional vector. Let $W_i = W(X_i)$ and $\bar W_n = \sum_{i=1}^n W_i / n$. Then $T_n = \bar W_n^T \bar W_n$. Let $\cI=\{(i, j), 1 \le i, j \le p\}$; let the random vector $U = (U_1, \ldots, U_p)^T$ be identically distributed as $X_i$. Then the covariance matrix $\Gamma = (\gamma_{a, a'})_{a, a' \in \cI}$ for $W = W(U)$ is $p^2 \times p^2$ with entries 
\begin{eqnarray*}
\gamma_{(i,j), (k,l) } &=& \E( (U_i U_j - \sigma_{i j})  (U_k U_l - \sigma_{k l}) )  \cr
 &=& \E(U_i U_j U_k U_l) - \sigma_{i j} \sigma_{k l} \cr
 &=& {\rm cum}(U_i, U_j, U_k, U_l)  + \sigma_{i k} \sigma_{j l} + \sigma_{i l} \sigma_{j k}.
\end{eqnarray*}
Let $U^*$ and $U$ be i.i.d. and $W^* = W(U^*)$. Observe that
\begin{eqnarray}
\label{eq:M210755p}
W^T W &=& (U^T U)^2 - 2 U^T \Sigma U + f^2, \cr
W^T W^* &=& (U^T U^*)^2  - U^T \Sigma U - U^{*T} \Sigma U^*  + f^2. 
\end{eqnarray}
In the sequel we shall deal with conditions (\ref{cond_1}) and (\ref{cond_2}) for the process $W_i = W(X_i)$ for $X_i$ satisfying (\ref{linear}). Lemma \ref{lem:denominator} provides a lower bound for $f^2_\vW  = \tr(\Gamma^{2}) = |\E (W W^T)|_F^2$, and Theorem \ref{thm:verify_cond_quad} leads to a bound for the quantities $K_\delta$ and $D_\delta$ for the $W$ vector.

\begin{lem} \label{lem:denominator}
Let $\nu=\Var\rbr{\xi_1^2}$. For $(\vX_{i})$ in (\ref{linear}), we have
\begin{align}\label{denominator}
f^2_\vW  := \sum_{a, b \in \cI} \gamma_{a b}^2 = \tr(\Gamma^{2}) 
 \ge \min(2, \nu^2/2)  f^4.
\end{align}
\end{lem}

To apply Theorem \ref{thm:true_sigma} on the random vectors $\vW = W(X)$; see (\ref{eq:M20813p}), we will need to find bounds $K_\delta^W$ and $D_\delta^W$ so that
\begin{align}\label{cond_quad_form}
\E\left|{\rBR{\vW}_2^2-\tr\rbR{\E\rbr{\vW\vW^T}}\over f_\vW}\right|^{2+\delta}&\le (K^W_\delta)^{2+\delta}, \\
\label{cond2_quad_form}
\E\left|{\vW^T\vW^*\over f_\vW}\right|^{2+\delta}&\le (D^W_\delta)^{2+\delta}.
\end{align}

By Lemma \ref{lem:denominator} and Theorem \ref{thm:verify_cond_quad} below, if $\xi_i$'s are not Bernoulli(1/2), we can have explicit bounds for $K_\delta^W$ and $D_\delta^W$. 

\begin{thm}\label{thm:verify_cond_quad}
Let $\vW_i = \vW(A\vxi_i)$. Suppose $\RBR{\xi_1}_{4 q} < \infty$, where $q = 2+\delta$ and $\delta > 0$. Let $\bar C_\delta = 2(4 q  \|\xi_1^2\|_{2q})^{2 q}$ and $\bar D_\delta = (4 q)^q \RBR{\xi_{1}}_{2q}^{2q} +(2q)^{2q}  \RBR{\xi_{1}}_{2q}^{4q}$.
Then
\begin{align}\label{numerator_cond_1}
\E\rBR{\rBR{\vW_1}_2^2-\tr\rbR{\E\rbr{\vW_1 \vW_1^T}}}^{2+\delta}&\le \bar C_\delta \rbr{ f_1f}^{2+\delta},\\
\label{numerator_cond_2}
\E\left|\vW_1^T\vW_2\right|^{2+\delta}&\le \bar D_\delta f^{4+2\delta}.
\end{align}
Thus if $\nu>0$, let $\theta = \min(2, \nu)/\sqrt 2$, then (\ref{cond_quad_form}) and (\ref{cond2_quad_form}) hold with $K^W_{\delta} = (\bar C_\delta/\theta)^{1/q} f_{1}/f$ and $D^W_{\delta} =  \bar D_\delta/\theta$, respectively. 
\end{thm}

\begin{rmk}
{\rm
A careful check of the proof of Theorem \ref{thm:verify_cond_quad} indicates that (\ref{numerator_cond_2}) holds under the milder moment condition $\xi_i \in {\cal L}^{4+2\delta}$. Instead of using $T_n$ in (\ref{eq:M20818p}), in view of (\ref{eq:M220648p}) we introduce the following quantity
\begin{align}\label{eq:M220650p}
\tilde T_{n}= {1\over{n(n-1)}} \sum_{i \not= i'} \sum_{j, k \le p}
 (X_{i j} X_{i k} -\sigma_{j k})   (X_{i' j} X_{i' k} -\sigma_{j k})
\end{align}
By (\ref{eq:M220656p}), under $\xi_i \in {\cal L}^{4+2\delta}$, we have
\begin{eqnarray*}
\sup_t\left|\P(n \tilde T_n \le f_W t)
      -\P\left( \sum_{a \in {\cal A}} { {\theta_a} \over f_W} (\eta_a-1) \le t \right)\right|
 = O(n^{- \delta/(10+4\delta)}),
\end{eqnarray*}
where $\theta_a$ are eigenvalues of $\Gamma$ and $\eta_a$ are i.i.d. $\chi^2_1$. \citet{MR2724863} consider testing the hypothesis $H_0: \Sigma = {\rm Id}_p$ vs $H_1: \Sigma \not= {\rm Id}_p$. They obtained a central limit theorem for a test statistic closely related to $\tilde T_{n}$ under the stronger moment assumption that $\eta_i$ has finite $8$th moment and $f_4 / f \to 0$. Our results relaxes the moment condition and can lead to a non-central limit theorem in that the asymptotic distribution may not be Gaussian. Additionally we have the rate of convergence of the approximate distribution. 
}
\end{rmk}

\section{A simulation study}
In this section we will provide a simulation study for the finite sample performances of the invariance principle Theorem \ref{thm:true_sigma}, the plug-in and the subsampling procedures described in Sections \ref{sec:O230147p} and \ref{sec:O230149p}, respectively. We consider the following two data generating models.
 
{\bf Model 1 (Linear Process):} Let  $\xi_{i,k}, i, k\in \mathbb Z$ are i.i.d. Student $t_5$; let
\begin{eqnarray}\label{eq:N080228p}
X_{i,j}=\sum_{k=0}^\infty (k+1)^{-\beta}\xi_{i,j-k}, \mbox{ where } \beta > 1/2.
\end{eqnarray}
If $\beta<1$, then the process $(X_{i, j})_j$ is long memory, thus having strong cross-sectional dependence. In our simulations we choose $p=200$ and $n=50, 200$ and truncate the sum in (\ref{eq:N080228p}) to $\sum_{k=0}^{2000}$, and choose two levels of $\beta$: $\beta=2$ and $\beta=0.6$, which correspond to short and long memory, respectively.

{\bf Model 2 (Factor Model):} Let 
\begin{eqnarray}
X_{i,j}=\sqrt{4+U_i^2}\xi_{i,j}+a (2Z_i+ Z_i^2-1), \, 1\le i \le n, 1\le j \le p,
\end{eqnarray}
where $U_i\sim Uniform[-1,1]$, $\xi_{i,j}, Z_i\sim N(0,1)$ and they are all independent. We consider two cases: $a = 0.05$ and $a = 0.5$, which imply weak and strong factors, respectively. We also let $p=200$ and $n=50, 200$.

We shall use QQ plots to measure the closeness of the approximations. Recall (\ref{eq:A291103p}) for $V$. Figures \ref{fig:srd}(a)-\ref{fig:nonlinearstrong}(a) show the QQ plots of the distributions of $R_n$ and $V$. In the literature majority of papers deal with central limit theorems for $R_n$. The normal QQ plots in Figures \ref{fig:srd}(b)-\ref{fig:nonlinearstrong}(b) indicate that the Gaussian approximation of $R_n$ can be quite bad if the cross-sectional dependence (among entries of $X_1$) is strong, see for example Model 1 with $\beta = 0.6$ and Model 2 with $a = 0.5$. In Figures \ref{fig:srd}(c)-\ref{fig:nonlinearstrong}(c), we make QQ plots for $\hat V$ vs $\hat R_n$. Here $\hat V = \sum_{j=1}^p \hat f^{-1} \hat \lambda_j (\eta'_j-1)$, where $\eta'_j$ are i.i.d. $\chi^2_1$ random variables that are independent of $\X_{n}$ and $\hat \lambda_j$ are eigenvalues of the sample covariance matrix $\hat \Sigma =  (n-1)^{-1} \sum_{i=1}^n (X_i - \bar X_n) (X_i - \bar X_n)^T$, and $\hat R_n = (n |\bar X_n|_2^2 - \hat f_1) / f^\dagger$, where $\hat f_1 = {\rm tr}(\hat \Sigma)$, and $f^\dagger = [{\rm tr} (\hat \Sigma^2)-\hat f_1^2/n]^{1/2}$; see \cite{MR1399305}. To obtain (c), the following steps are repeated for $N=100$ times: in each realization, data is generated according to the above models. Then given $\hat \Sigma$, we obtain $K=100$ realizations of $\hat V$ by generating $100 p$ i.i.d. $\chi^2_1$ r.v. $\eta_j'$. Figures \ref{fig:srd}(c)-\ref{fig:nonlinearstrong}(c) suggest that, for the plug-in procedure, larger $n$ leads to better approximations.  Figures \ref{fig:srd}(d)-\ref{fig:nonlinearstrong}(d) show the subsampling procedure (cf. Theorem \ref{th:O240914a}(ii)). As in (c), we perform in (d) the QQ plots of $N=100$ repetitions of $n |\bar X_n|_2^2$ and the subsample values $m(1-m/n)^{-1}|\bar X_{A_j}-\bar X|_2^2$ with $J=100$ and $m = \lfloor n/\log n\rfloor$. The subsampling distribution provides an excellent approximation of the distribution of $n |\bar X_n|_2^2$. For the subsampling approach one needs to choose an $m$. In our simulation study for other models (not reported here) with bounded $K_\delta$ and $D_\delta$, the rule-of-thumb choice $m = \lfloor n/\log n\rfloor$ can often have a satisfactory performance. We leave it as a future problem on designing a data-driven choice of $m$.

\begin{figure}[htbp] 
   \centering
   \includegraphics[width=6in]{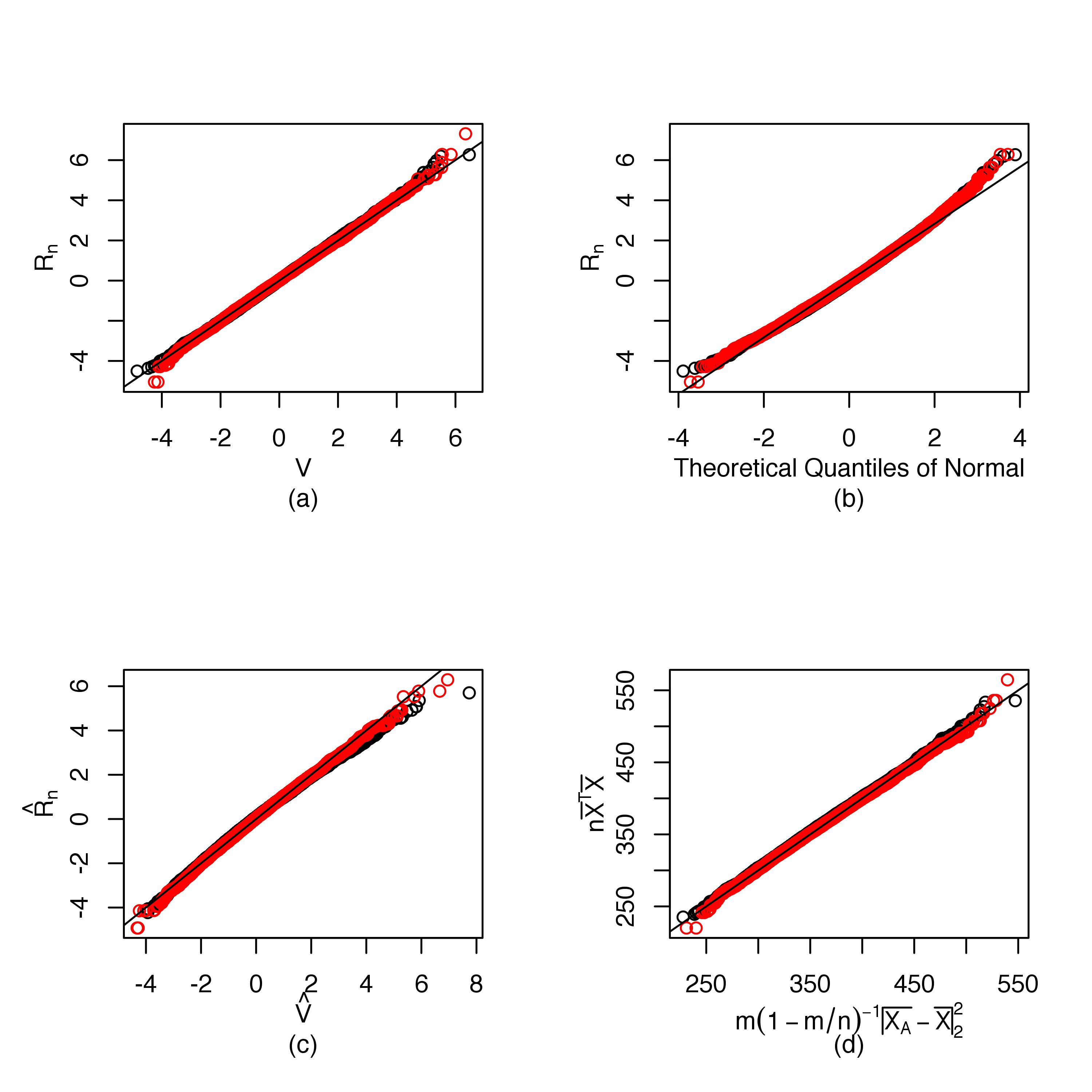} 
   \caption{Model 1 with $\beta=2$. (a) QQ-plot of $V$ v.s. $R_n$ (cf. Theorem \ref{thm:true_sigma}); (b) QQ-normal plot of $R_n$; (c) QQ-plot of $\hat V$ v.s. $\hat R_n$; (d) QQ-plot of the subsampling distribution v.s. $n |\bar X|_2^2$ (cf. Theorem \ref{th:O240914a}(ii)). Red: $n = 200$; black: $n = 50$.}
   \label{fig:srd}
\end{figure}
\begin{figure}[htbp] 
   \centering
   \includegraphics[width=6in]{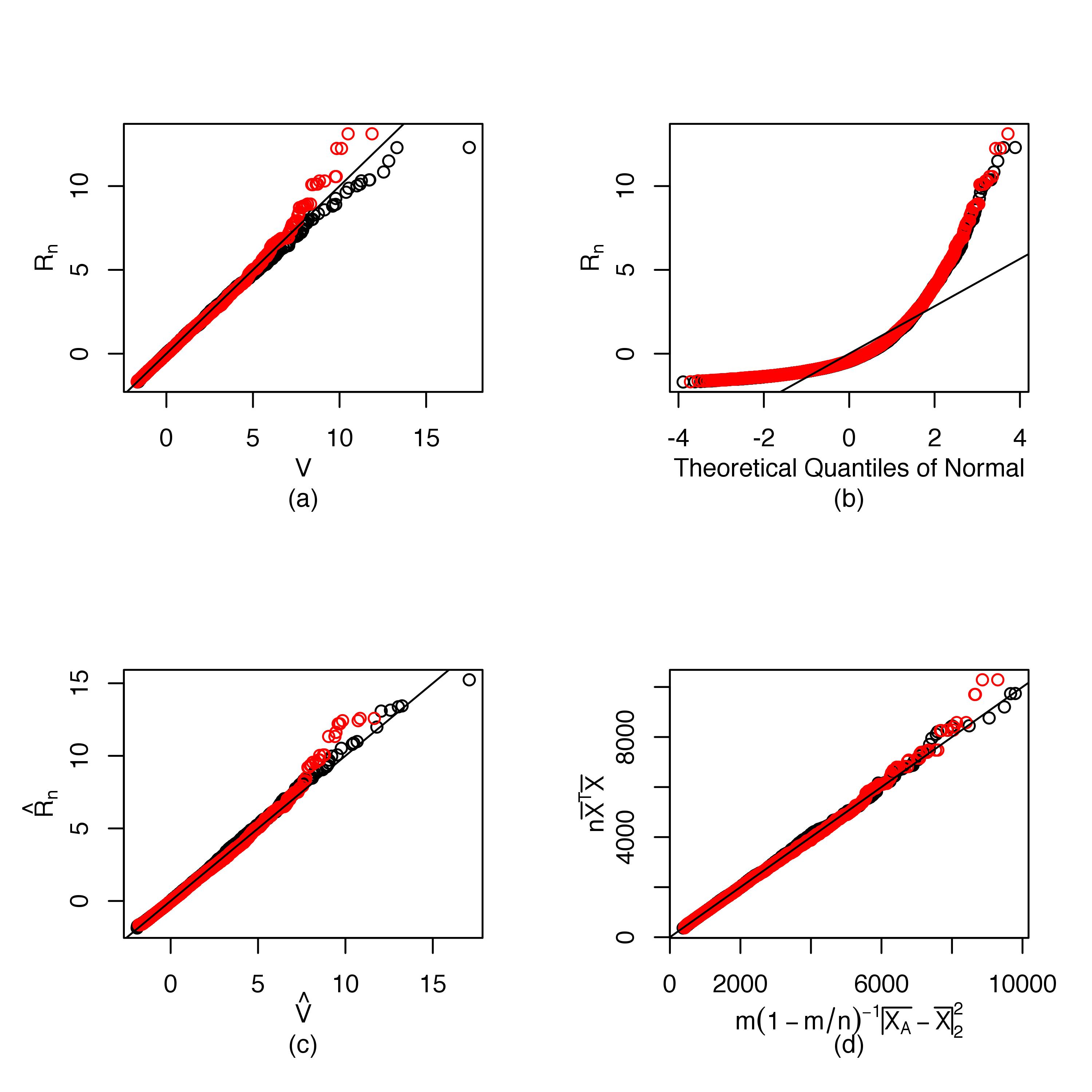} 
   \caption{Model 1 with $\beta=0.6$.  See Figure \ref{fig:srd} for the caption.}
   \label{fig:lrd}
\end{figure}
\begin{figure}[htbp] 
   \centering
   \includegraphics[width=6in]{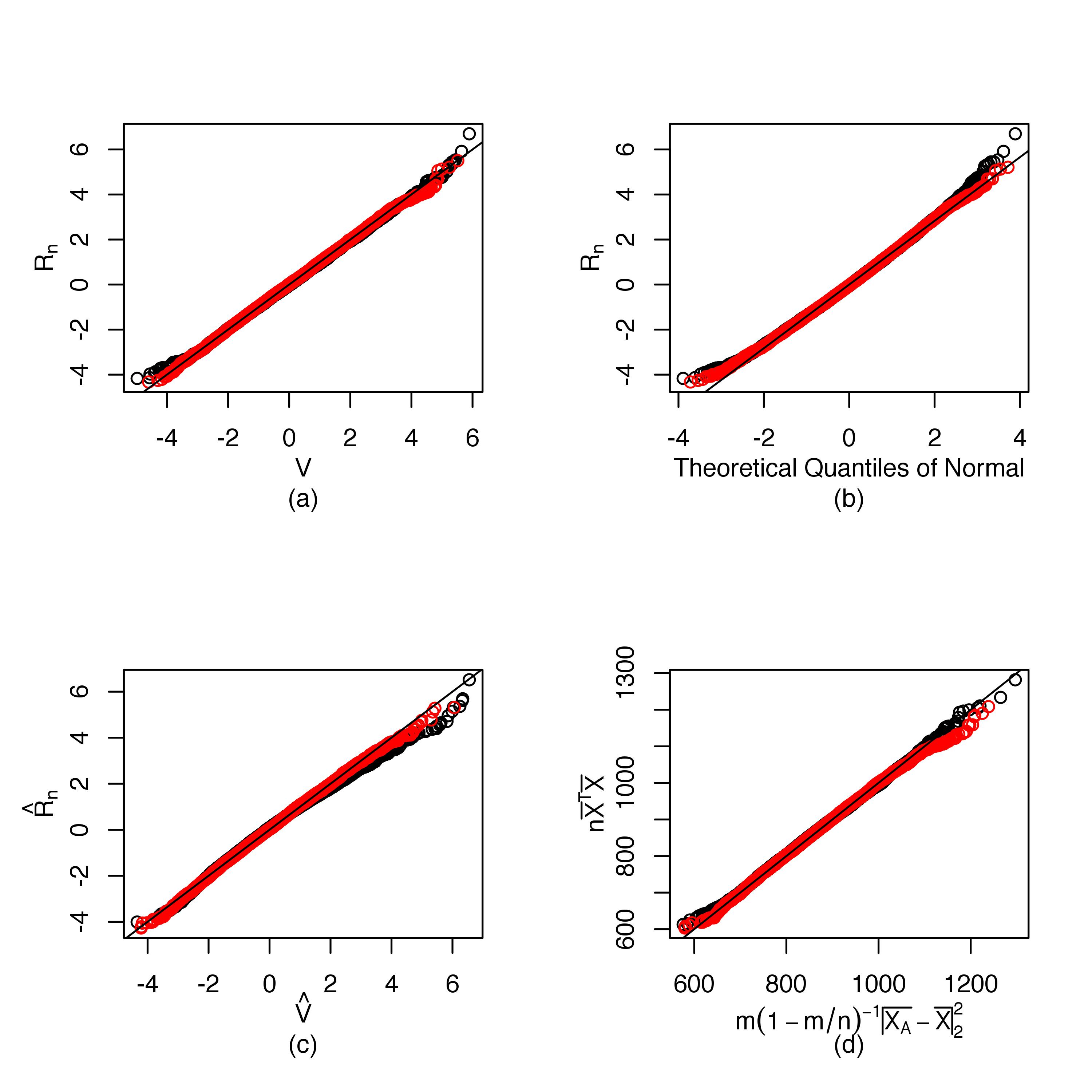} 
   \caption{Model 2 with $a = 0.05$. See Figure \ref{fig:srd} for the caption.}
   \label{fig:nonlinearweak}
\end{figure}
\begin{figure}[htbp] 
   \centering
   \includegraphics[width=6in]{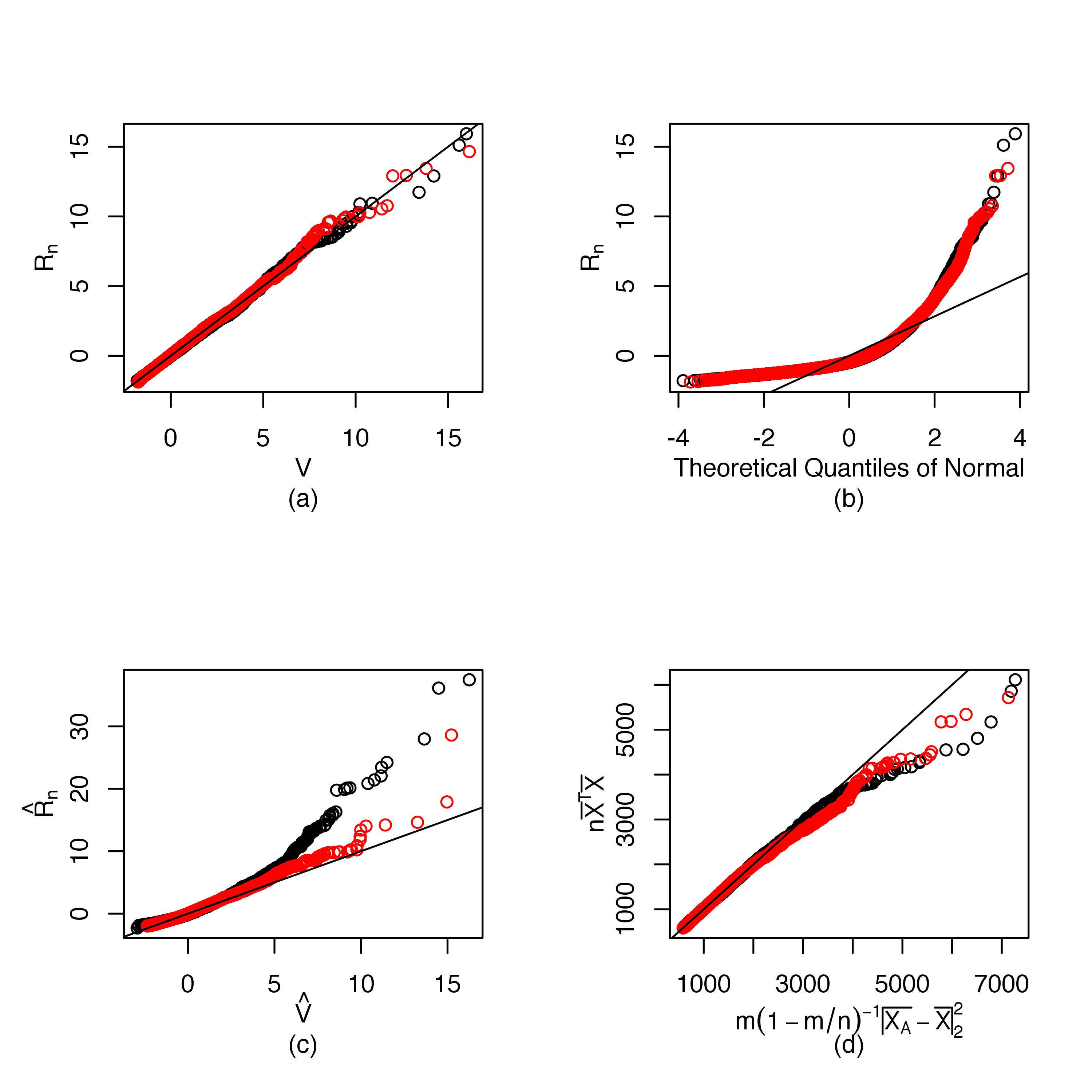} 
   \caption{Model 2 with $a = 0.5$. See Figure \ref{fig:srd} for the caption.}
   \label{fig:nonlinearstrong}
\end{figure}

\section{Proof}
\label{sec:proof}

\begin{proof}[Proof of Proposition \ref{prop:cond_Y}]
Note that $\vxi=\Lambda^{-1/2}Q^T\vY_1 \sim N(\vzero, {\rm Id}_p)$. Then $\vY_1^T \vY_1 = \sumj \lambda_{j}\xi_j^2$, where $\xi_j$ are entries of $\vxi$ and are i.i.d. $N(0, 1)$. Let $q = 2 + \delta$. By Burkholder's inequality (\citet{MR1476912}), 
\begin{eqnarray*}
\RBR{\vY^T\vY-f_1}_q^2 \le (q-1)  \sum_{j=1}^p \lambda_j^2 \|\xi^2_j-1\|_q^2.
\end{eqnarray*}
Then  (\ref{cond_1Y}) holds. Let $\vzeta=\Lambda^{-1/2}Q^T\vY_2$. Then $\vY_1^T \vY_2 = \sumj \lambda_{j}\xi_j\zeta_j$ and (\ref{cond_2Y}) similarly follows. 
\end{proof}

In Lemma \ref{lem:g_approx} and in the proof of Theorem \ref{thm:true_sigma}, we define 
\begin{eqnarray}\label{eq:M01725p}
g_0(u) = (1- \min(1, \max(u, 0))^4)^4.
\end{eqnarray}
Any non-increasing function $g_0(\cdot)$ with $g_0(\cdot) \in \mathbb{C}^3$, $g_0(u) = 1$ if $u \le 0$, and $g_0(u)=0$ if $u \ge 1$, will meet our requirements. To make the calculations explicit, we can choose $g_0$ in the form of (\ref{eq:M01725p}). Then
\begin{align}\label{eq:M01726p}
g_* = \max_u [|g'_0(u)| + |g''_0(u)| + |g'''_0(u)|] < \infty.
\end{align}
\begin{proof}[Proof of Theorem \ref{thm:true_sigma}]
Let $\vY_i \in \mathbb{R}^p$ be i.i.d. $N(\mathbf{0}, \Sigma)$ random vectors and $\vYbar = \sumn\vY_i / n$. Then $\sum_{j=1}^p {\lambda_j} \eta_i$ and $n |\bar Y_n|_2^2$ are identically distributed. Note that $f_1 =  \sum_{j=1}^p {\lambda_j}$. Hence, to show (\ref{eq:A291103p}), since $L_\delta(n, \psi)$ is increasing in $\psi$, it suffices to prove the following relation holds for every $\psi$:
\begin{align}\label{gar2}
\sup_t\left|\P\left( R_n \le t\right)-\P\left( R^\diamond_n \le t\right)\right|= O(L_\delta(n, \psi)+\psi^{-1/2}),
\end{align}
where $R^\diamond_n$ is the Gaussian version of $R_n$ in (\ref{eq:M180141p}):
\begin{eqnarray}
\label{eq:A300741p}
R^\diamond_n = { {n |\bar Y_n|_2^2 - f_1} \over f}.
\end{eqnarray}

Recall (\ref{eq:M01725p}) for $g_0$. We first approximate the indicator function $h(x)=\ind{x\le t}$ the $\mathbb{C}^3$ function $\g(x)=g_0(\psi(x-t))$ for $t$ fixed. By (\ref{eq:M01726p}), 
\begin{align*}
&\ind{x\le t}\le \g(x)\le \ind{x\le t+\psi^{-1}},\\
\quad \sup_{x,t}|\g'(x)|& \le g_* \psi,\,\, \sup_{x,t}|\g''(x)| \le g_* \psi^2,\,\,\sup_{x,t}|\g'''(x)| \le g_* \psi^3.
\end{align*}
Then $\P\left(R_n \le t\right) \le \E\g\left(R_n \right)$. By Lemma \ref{lem:g_approx}, 
\begin{eqnarray}\nonumber
\E\g\left( R_n \right)
&\le&\E\g\left( R^\diamond_n \right)+CL_\delta(n,\psi)\\\label{right_side}
&\le&\P\left( R^\diamond_n \le t+\psi^{-1}\right)+CL_\delta (n,\psi).
\end{eqnarray}
The reverse direction is similar: by applying Lemma \ref{lem:g_approx} again, we have
\begin{eqnarray}
\label{left_side}
\P\left(R_n \le t\right) \ge \P\left( R^\diamond_n \le t-\psi^{-1}\right) - C L_\delta(n,\psi).
\end{eqnarray}
By (\ref{right_side}), (\ref{left_side}) and (\ref{holder}) in Lemma \ref{lem:density}, we have (\ref{gar2}).
\end{proof}


\begin{lem}\label{lem:g_approx}
Assume (\ref{cond_1}) and (\ref{cond_2}). Let $\tilde K_\delta$ and $\tilde D_\delta$ be specified as in Theorem \ref{thm:true_sigma}. Let $\g(x) = g_0(\psi(x-t))$, where $g_0(\cdot)$ is given by (\ref{eq:M01725p}). Recall (\ref{eq:A300741p}) for $R_n$ and $R^\diamond_n$. Then we have
\begin{eqnarray}\label{eq:A30814p}
\sup_t\left|\E\g\left( R_n \right)-\E\g\left(R^\diamond_n \right)\right|
  = O[L_\delta(n, \psi)].
\end{eqnarray}
\end{lem}
\begin{proof}[Proof of Lemma \ref{lem:g_approx}]
Let $H_i = \sum_{j=1}^{i-1} \vX_j + \sum_{j=i+1}^n \vY_j$ and
\begin{align*}
&L_i=\frac{H_i^TH_i-(n-1)f_1}{nf},\\
&\Delta_i=\frac{2H_i^T\vX_i+\vX_i^T\vX_i-f_1}{nf},\\
&\Gamma_i=\frac{2H_i^T\vY_i+\vY_i^T\vY_i-f_1}{nf}.
\end{align*}
Note that $H_i$ is independent of $\vX_i$ and $\vY_i$. Let
\begin{align*}
\I&= \g'(L_i)(\Delta_i-\Gamma_i),\\
\II&=\frac{1}{2}\g''(L_i)(\Delta_i^2-\Gamma_i^2),\\
\III&=\left[\g\left(L_i+\Delta_i\right)-\g\left(L_i+\Gamma_i\right)\right]-\I-\II.
\end{align*}
Note that $\vX_i$ and $\vY_i$ both have mean $\mathbf{0}$ and covariance matrix $\Sigma$. Then
\begin{align*}
\E\I&=\E\E\left[\g'(L_i)(\Delta_i-\Gamma_i)\big|\vX_i,\vY_i\right]\\
&=\frac{1}{nf}\E\left[2(\vX_i^T-\vY_i^T)\E(\g'(L_i)H_i)+(\vX_i^T\vX_i-\vY_i^T\vY_i)\E \g'(L_i)\right] = 0.
\end{align*}
For $\II$, by (\ref{eq:M01726p}), $|\g''(u)| \le g_* \psi^2$. Then for $C_1 = g_* / 2$,
\begin{align*}
\rBR{\E \II}& = \rBR{\frac{1}{2}\E[\g''(L_i)(\Delta_i^2-\Gamma_i^2)]}\\
&=\frac{1}{2} \rBR{\E\rbR{\g''(L_i)\E\rbr{\cond{\Delta_i^2-\Gamma_i^2}H_i}}}\\
&\le C_1 \psi^2\E\rBR{\E\rbr{\left.\Delta_i^2-\Gamma_i^2\right|H_i}}.
\end{align*}
The term $n^2f^2\E\rbr{\left.\Delta_i^2-\Gamma_i^2\right|H_i}$ can be decomposed into
\begin{align*}
&4 \E\rbR{\cond{H_i^T\vX_i\vX_i^TH_i-H_i^T\vY_i\vY_i^TH_i}H_i}
+\E\rbR{(\vX_i^T\vX_i-f_1)^2-(\vY_i^T\vY_i-f_1)^2}\\
&\qquad+4\E\rbR{\cond{H_i^T\vX_i(\vX_i^T\vX_i-f_1)-H_i^T\vY_i(\vY_i^T\vY_i-f_1)}H_i},
\end{align*}
where $\E\rbr{\cond{H_i^T\vX_i\vX_i^TH_i-H_i^T\vY_i\vY_i^TH_i}H_i}=0$. By (\ref{cond_1}),
\begin{align*}
\E\left|(\vX_i^T\vX_i-f_1)^2-(\vY_i^T\vY_i-f_1)^2\right| \le f^2 (K_{0}^2+c_{0}^2) \le f^2 \tilde K_{0}^2.
\end{align*}
Since $Y_i$ is Gaussian, $\E\rbR{\cond{H_i^T\vY_i(\vY_i^T\vY_i-f_1)}H_i}=0$. By the Cauchy-Schwarz inequality and (\ref{cond_1}), since $\| H_i^T \vX_i \|^2 = (n-1)\tr(\Sigma^2) = (n-1) f^2$,
\begin{eqnarray*}
 \E\rBR{\E\rbr{\left.\Delta_i^2-\Gamma_i^2\right|H_i}}
 &\le& { {\tilde K_{0}^2 } \over n^2}
  + { {\E |\E [\cond{H_i^T\vX_i(\vX_i^T\vX_i-f_1)}H_i ] |} \over {n^2 f^2} }\cr
  &\le& { {\tilde K_{0}^2} \over n^2}
  + { {\| H_i^T\vX_i \|
   \| \vX_i^T\vX_i-f_1 \|   } 
     \over {n^2 f^2} } \cr
 & \le& { { \tilde K_{0}^2} \over n^2} 
  + { {K_{0} } \over {n^{3/2} } }.
\end{eqnarray*}
So
\begin{align}\label{eq:A30816p}
\rBR{\E\II}\le C\psi^2 (n^{-2}\tilde K_{0}^2+n^{-3/2}\tilde K_{0}).
\end{align}
Since $0\le g(t)\le 1$ for all $t$, and  $|\g'''(u)| \le g_* \psi^3$. We have that
\begin{align*}
\E\rBR{\III}
&\le\E\min\rBr{1+|\I|+|\II|, g_* \psi^3 (|\Delta_i|^3+|\Gamma_i|^3)}\\
& \le C\E\min\rBr{1+\psi(|\Delta_i|+|\Gamma_i|)+\psi^2(|\Delta_i|^2+|\Gamma_i|^2),\psi^3(|\Delta_i|^3+|\Gamma_i|^3)}\\
&\le C\psi^q (\E\rBR{\Delta_i}^q+\E\rBR{\Gamma_i}^q),
\end{align*}
where $q = 2+\delta$. Let $\vx \in {\mathbb{R}^p}$ be a fixed vector. By Rosenthal's inequality,
\begin{eqnarray}\label{eq:A30911p}
\E \left| H_i \vx\right|_q^q \le c_q [m \|X_1^T\vx \|_q^q
 + (n-m) \|Y_n^T \vx \|_q^q
 + n^{q/2} (\vx^T \Sigma \vx)^{q/2} ],
\end{eqnarray}
where $c_q$ and $c_{q,1}, \ldots$ hereafter are constants only depend on $q$ and they may take different values at different appearances. Note that $Y_n^T \vx \sim N(0, \vx^T \Sigma \vx)$. Let $c_{q,1} = \| \xi_1 \|^q_q$, $\xi_1 \sim N(0, 1)$. Then $\E |Y_n^T \vx|^q = c_{q, 1} (\vx^T \Sigma \vx)^{q/2}$ and  
\begin{eqnarray}
\| H_i^T\vX_i \|_q^q \le c_q(n\RBR{X_1^TX_2}_q^q+ n^{q/2}  \E (X_1^T \Sigma X_1)^{q/2})
\end{eqnarray}
Hence by (\ref{cond_2}) and (\ref{cond_2Y}), we have
\begin{eqnarray}
\label{eq:A30935p}
\E|\Delta_i|^q
&\le& C {{\E|H_i^T\vX_i|^q+\E|\vX_i^T\vX_i-f_1|^q} \over {n^q f^q} } \cr
&\le& C {{ n\tilde D_\delta^qf^q+ n^{q/2}  \E (X_1^T \Sigma X_1)^{q/2} + K_\delta^q f^q} \over {n^q f^q} }. 
\end{eqnarray}
By (\ref{eq:A30911p}), $\| H_i^T\vY_i \|_q^q \le c_q (n \E (X_1^T \Sigma X_1)^{q/2}  + n^{q/2} f^q)$, which implies that $\E\rBR{\Gamma_{i}}^q\le c_q ( n \E (X_1^T \Sigma X_1)^{q/2} / (n f)^q + n^{-q/2})$. Observe that $(|H_i+\vX_i|_2^2 - n f_1)/(n f) = L_i+\Delta_i$ and $(|H_i+\vY_i|_2^2 - n f_1)/(n f) = L_i+\Gamma_i$. We write the telescope sum
\begin{align*}
\g\left( R_n \right) - \g\left(R^\diamond_n \right) = \sumn\left[\g(L_i+\Delta_i) - \g(L_i+\Gamma_i)\right],
\end{align*}
which entails (\ref{eq:A30814p}) in view of (\ref{eq:A30816p}), (\ref{eq:A30935p}) and $\E\I = 0$.
\end{proof}

\begin{lem}\label{lem:density}
Let $a_1 \ge \ldots \ge a_p \ge 0$ be such that $\sum_{i=1}^p a_i^2 = 1$; let $\eta_i$ be i.i.d. $\chi^2_1$ random variables. Then for all $h > 0$, 
\begin{align}\label{holder}
\sup_t \P\left(t\le a_1 \eta_1 + \ldots + a_p \eta_p \le t+h\right) \le h^{1/2} \sqrt{4/\pi}.
\end{align}
\end{lem}

\begin{proof}[Proof of Lemma \ref{lem:density}]
Write $V = \sum_{i=1}^p a_i \eta_i$. Assume $a_1 \le 1/2$. Then its characteristic function $\phi_V(s) = \E \exp(\sqrt{-1} s V), s \in \mathbb{R}$, satisfies
\begin{eqnarray}
\label{eq:A300350p}
|\phi_V(s)| &=& \left| \prod_{j=1}^p (1-2 \sqrt{-1} a_{j} s)^{-1/2} \right| \cr
 &=& \prod_{j=1}^p(1+4a_{j}^2s^2)^{-1/4} \cr
 &\le& (1 + 4 s^2 + 8 b_4 s^4 + 32/3 b_6 s^6)^{-1/4},
\end{eqnarray}
where $b_4 = \sum_{j\ne k} a_j^2 a_k^2 = 1 - \sum_{k=1}^p a_k^4 \ge 1-a_1^2 \ge 3/4$ and
\begin{eqnarray*}
b_6 &=& \sumstar_{j,k,l} a_j^2 a_k^2 a_l^2 = 1-3\sum_{j\ne k} a_j^4a_k^2 - \sumj a_j^6 \cr
 &\ge& 1-3\sumj a_j^4 \left(\sum_{k\ne j}a_k^2+a_j^2\right) \ge 1 - 3 a_1^2 \ge 1/4.
\end{eqnarray*}
By the inversion formula and (\ref{eq:A300350p}), the density function $f_V(\cdot)$ of $V$ satisfies
\begin{eqnarray*}
f_V(v) = {1\over {2 \pi}} \int_{-\infty}^\infty e^{-\sqrt{-1} v s} \phi_V(s) d s
 \le  {1\over {2\pi}}  \int_{-\infty}^\infty  |\phi_V(s)| d s < 1
\end{eqnarray*}
Now we shall deal with the case that $a_1 > 1/2$. Note that for all $w > 0$, $\sup_u \P(u \le \eta_1 \le u+w) \le w^{1/2} \sqrt{2 / \pi}$. Then $\sup_t \P\left(t\le V \le t+h\right) \le (2 h)^{1/2} \sqrt{2/\pi}$. Combining with the case $a_1 \le 1/2$, we obtain the upper bound $\max(h^{1/2} \sqrt{4/\pi}, \, h)$. Note that (\ref{holder}) trivially holds if $h \ge 1$. 
\end{proof}

\begin{proof}[Proof of Proposition \ref{prop:M020905a}]
Note that $\rho(\Sigma / f) \le |\Sigma / f|_F = 1$. Since $\tilde f / f - 1 = o_\P(1)$, $\rho(\Sigma/f) (f/\tilde f - 1) = o_\P(1)$. Hence for the "if" part,
\begin{eqnarray*}
\rho({\tilde \Sigma / \tilde f - \Sigma/f}) 
 \le \rho({\tilde \Sigma - \Sigma}) / \tilde f + \rho(\Sigma/f) |f/\tilde f - 1| = o_\P(1)
\end{eqnarray*}
The "only if" part can be similarly proved. 
\end{proof}
\begin{proof}[Proof of Lemma \ref{lem:com_mix_chisq}]
Let $\rho_p = \max_j\rBR{a_{p,j}-b_{p,j}}$. Choose an integer sequence $K = K_p$ such that $K_p \to \infty$ and $K_p \rho_p \to 0$.  Let $W = \sum_{j=1}^{K-1} a_{p,j} \eta'_j$, $W^\circ = \sum_{j=K}^p a_{p,j}\eta'_j$,  $S = \sum_{j=1}^{K-1} b_{p,j} \eta'_j$, $S^\circ = \sum_{j=K}^p b_{p,j} \eta'_j$, $w = 2 \sum_{j=K}^p a_{p, j}^2$ and $s = 2 \sum_{j=K}^p b_{p, j}^2$.  Let $u_K = a_{p, K}^{1/4}$. By the Gaussian approximation result in \cite{MR2302850}, on a richer probability space, we can construct a random variable $Z \sim N(0, 1)$, independent of $(\eta_i)_{i=1}^{K-1}$, such that 
\begin{eqnarray}
\P( |W^\circ - w^{1/2} Z| \ge u_K)  \le {c_4 \over u_K^4} \sum_{j=K}^p a_{p,j}^4
 \le {c_4 \over u_K^4} a_{p, K}^2 = c_4 u_K^4,
\end{eqnarray}
where $c_4 > 0$ is an absolute constant. Since $u_K \to 0$, by Lemma \ref{lem:density}, 
\begin{eqnarray}\label{eq:M20349p}
\sup_x| \P( |W + W^\circ| \le x) - \P( |W + w^{1/2} Z| \le x)| \to 0.
\end{eqnarray}
Similarly, for $v_K = b_{p, K}^{1/4}$, we can also construct a probability space with a r.v. $Z^* \sim N(0, 1)$ such that $\P( |S^\circ - w^{1/2} Z^*| \ge v_K)  \le c_4 v_K^4$, and 
\begin{eqnarray}\label{eq:M20350p}
\sup_x| \P( |S + S^\circ| \le x) - \P( |S + s^{1/2} Z^*| \le x)| \to 0.
\end{eqnarray}
Let $T = (W + w^{1/2} Z) - (S + s^{1/2} Z)$. Since $w-s = 2 \sum_{j=1}^{K-1} (b_{p,j}^2-a_{p,j}^2)$, 
\begin{eqnarray}
\E |T| &\le& 2 (K-1) \rho_p + |w^{1/2} - s^{1/2}| \cr
  &\le& 2 K \rho_p + |w-s|^{1/2} \le  2 K \rho_p + (4 K \rho_p)^{1/2} \to 0. 
\end{eqnarray}
Hence, by (\ref{eq:M20349p}), (\ref{eq:M20350p}) and Lemma \ref{lem:density}, (\ref{com_mix_chisq}) follows. 
\end{proof}
\begin{proof}[Proof of Theorem \ref{thm:spec_Sigma_over_f}]
(i) Since $X_i$ are i.i.d., we have
\begin{align*}
\E |\hat\Sigma-\Sigma|_F^2&=\E\sumjk (\hat\sigma_{jk}-\sigma_{jk})^2\\
&={1\over n}\sumjk\E\rbr{X_{1j}^2X_{1k}^2-\sigma_{jk}^2}\\
&={1\over n}\E\Big[{\big({\sumj X_{1j}^2}\big)^2}\Big]-{1\over n}f^2 \\
& = {1\over n} \E[(\vX_1^T\vX_1)^2] -{1\over n}f^2,
\end{align*}
which, by the assumption $\E[(\vX_1^T\vX_1)^2]=o\rbr{n f^2}$, implies $\E |\hat\Sigma-\Sigma|_F^2 = o(f^2)$. Then $\| |\Sigma|_F - |\hat \Sigma|_F \|_{2}\le \| |\hat\Sigma-\Sigma|_F \|_{2} = o(f)$, or $\| f - \hat f\|_{2} = o(f)$, and 
\begin{align*}
\| \hat\Sigma / \hat f - \Sigma / f\|_F 
 \le \|(\hat\Sigma-\Sigma)/ f \|_F
 + \| |\hat\Sigma/ \hat f|_F | 1-\hat f/f| \|= o(1).
\end{align*}

(ii) Let $g = \| \vX_{1}^T\vX_{1} \|$. Since $n f^2 = o(g^2)$, by Schwarz's inequality,
\begin{eqnarray}\label{high_moment3}
\E[(\vX_1^T\Sigma\vX_1)^2] &\le& \E(\vX_1^T\Sigma^2\vX_1\vX_1^T\vX_1)
 =  \E\tr[(\vX_1\vX_1^T)^2\Sigma^2] \cr
 &\le& \E[\sqrt{\tr(\Sigma^4)}(\vX_{1}^T\vX_{1})^2] 
 \le f^2 g^2 = o({g^4 \over n}).
\end{eqnarray}
Since (\ref{cond_1}) holds with $K_2=O(n^{3/4})$ and $\E(\vX_1^T \vX_1) = f_1 \le g$,  we have 
\begin{eqnarray}\label{high_moment4}
\RBR{\vX_{1}^T\vX_{1}}_4^4
 \le 8\RBR{\vX_{1}^T\vX_{1}-f_1}_4^4+8f_1^4 
 \le 8K_2^4 f^4+8f_1^4 
 = o(n g^4).
\end{eqnarray}
By (\ref{high_moment4}) and (\ref{high_moment3}), we have  
\begin{eqnarray}
\label{high_moment5}
\E[(\vX_1^T\vX_1)^2\vX_1^T\Sigma\vX_1] 
 \le \{ {\E[(\vX_1^T\vX_1)^4]} {\E[(\vX_1^T\Sigma\vX_1)^2]} \}^{1/2}
 = o(g^4).
\end{eqnarray}
Since $\E[(\vX_1^T\vX_1)^2(\vX_1^T\vX_2)^2] = \E[(\vX_1^T\vX_1)^2\vX_1^T\Sigma\vX_1]$, by (\ref{high_moment5}), we have 
\begin{eqnarray}
\label{high_moment6}
\E[\vX_1^T\vX_1(\vX_1^T\vX_2)^2\vX_2^T\vX_2]
 \le\E[(\vX_1^T\vX_1)^2(\vX_1^T\vX_2)^2] 
 = o(g^4)
\end{eqnarray}

Since $(\rho(\hat \Sigma) / \hat f)^4 \le {\hat f_4^4 / \hat f^4}\le (\rho(\hat \Sigma) / \hat f)^{2}$, it suffices to show that $\hat f_4^4/\hat f^4=o_\P(1)$. Clearly the latter follows from
\begin{eqnarray}\label{eq:M180327p}
\E(\hat f_4^4) = o(\E^2(\hat f^2)) \mbox{ and }
 \E( \hat f^2 / \E(\hat f^2) - 1)^2 = o(1).
\end{eqnarray}

An expansion of $\hat f^2 = \sum_{j, k \le p} \hat \sigma_{j k}^2$ yields that
\begin{align*}
n^2 \E (\hat f^2)&= {\sum_{1\le i\ne l\le n,1\le j,k\le p}\E (X_{ij}X_{ik}X_{lj}X_{lk})
   +\sum_{1\le i\le n,1\le j,k\le p}\E (X^2_{ij}X^2_{ik})} \\
&= {n(n-1)\sumjk\sigma_{jk}^2+n\sumjk\E (X^2_{ij}X^2_{ik})} \\
&= (n^2-n) f^2+ n \E[(\vX_1^T\vX_1)^2].
\end{align*}
Since $n f^2 = o(g^2)$, we have $\E (\hat f^2) \asymp n^{-1} g^2$. Write
\begin{align*}
n^4 \E(\tr(\hat\Sigma^4))
&= \sum_{1\le j,k,m,q\le p}\sum_{1\le i,l,s,t\le n}\E\rbr{X_{ij}X_{ik}X_{lk}X_{lm}X_{sm}X_{sq}X_{tq}X_{tj}}\\
&=: \I+\II+\III+\IV+\V+\VI+\VII,
\end{align*}
where, based on the number of distinct indexes in $\{i,l,s,t\}$, 
\begin{align*}
\I =& n(n-1)(n-2)(n-3)\sum_{1\le j,k,m,q\le p}\sigma_{jk}\sigma_{km}\sigma_{mq}\sigma_{qj}\\
\II =&4n(n-1)(n-2)\sum_{1\le j,k,m,q\le p}\E\rbr{X_{1j}X_{1k}^2X_{1m}}\sigma_{mq}\sigma_{qj}\\
\III =&2n(n-1)(n-2)\sum_{1\le j,k,m,q\le p}\E\rbr{X_{1j}X_{1k}X_{1m}X_{1q}}\sigma_{km}\sigma_{qj}\\
\IV =&2n(n-1)\sum_{1\le j,k,m,q\le p}\E\rbr{X_{1j}X_{1k}^2X_{1m}}\E\rbr{X_{1m}X_{1q}^2X_{1j}}\\
\V =&n(n-1)\sum_{1\le j,k,m,q\le p}\rbR{\E\rbr{X_{1j}X_{1k}X_{1m}X_{1q}}}^2\\
\VI =&4n(n-1)\sum_{1\le j,k,m,q\le p}\sigma_{jk}\E\rbr{X_{1j}X_{1k}X_{1m}^2X_{1q}^2}\\
\VII =&n\sum_{1\le j,k,m,q\le p}\E\rbr{X_{1j}^2X_{1k}^2X_{1m}^2X_{1q}^2}.
\end{align*}
Note that $\tr(\Sigma^k/f^k)\le \rho(\Sigma/f)^{k-2}=o(1)$ for $k>2$. By (\ref{high_moment2}) and (\ref{high_moment3})--(\ref{high_moment6}), we obtain by elementary manipulations that $\E(\hat f_4^4) = o(\E^2(\hat f^2))$. To prove the second assertion of (\ref{eq:M180327p}), we similarly write
\begin{align*}
n^4 \E\hat f^4&= \sum_{1\le j,k,m,q\le p}\sum_{1\le i,l,s,t\le n}\E\rbr{X_{ij}X_{ik}X_{lj}X_{lk}X_{sm}X_{sq}X_{tm}X_{tq}}\\
&= \I'+\II'+\III'+\IV'+\V'+\VI'+\VII',
\end{align*}
where
\begin{align*}
\I':=& n(n-1)(n-2)(n-3)\sum_{1\le j,k,m,q\le p}\sigma_{jk}^2\sigma_{mq}^2\\
\II':=&2n(n-1)(n-2)\sum_{1\le j,k,m,q\le p}\E\rbr{X_{1j}^2X_{1k}^2}\sigma_{mq}^2\\
\III':=&4n(n-1)(n-2)\sum_{1\le j,k,m,q\le p}\E\rbr{X_{1j}X_{1k}X_{1m}X_{1q}}\sigma_{jk}\sigma_{mq}\\
\IV':=&4n(n-1)\sum_{1\le j,k,m,q\le p}\sigma_{jk}\E\rbr{X_{1j}X_{1k}X_{1m}^2X_{1q}^2}\\
\V':=&n(n-1)\sum_{1\le j,k,m,q\le p}\E\rbr{X_{1j}^2X_{1k}^2}\E\rbr{X_{1m}^2X_{1q}^2}\\
\VI':=&2n(n-1)\sum_{1\le j,k,m,q\le p}\rbR{\E\rbr{X_{1j}X_{1k}X_{1m}X_{1q}}}^2\\
\VII':=&n\sum_{1\le j,k,m,q\le p}\E\rbr{X_{1j}^2X_{1k}^2X_{1m}^2X_{1q}^2}.
\end{align*}
Then the second assertion of (\ref{eq:M180327p}) similarly follows from (\ref{high_moment2}), (\ref{high_moment3})--(\ref{high_moment6}).
\end{proof}


\begin{proof}[Proof of Theorem \ref{thm:verify_cond_quad}]
Write $\vW$ and $\vW^*$ for $W_1$ and $W_2$, respectively. Let $B=A^T A$ and $U =A\vxi$. Then $f_1 = \tr(B)$, $U^T \Sigma U = \vxi^T B^2 \vxi$ and $U^T U = \vxi^T B \vxi$. By the argument in (\ref{eq:M210757p}), we have $\RBR{U^T U - f_1}^2_q \le 4 q^2 \|\xi_1^2\|_{q}^2 f^2$,
\begin{eqnarray*}
\RBR{U^T \Sigma U - f^2}_q^{2} &\le& 
  4 q^2 \|\xi_1^2\|_q^2 \tr(B^4) \le 4 q^2 \|\xi_1^2\|_q^2 f^4, \cr
\RBR{(U^T U - f_1)^2}_q &=& \RBR{U^T U - f_1}_{2q}^2  \le 4 (2q)^2 \|\xi_1^2\|_{2q}^2 f^2
\end{eqnarray*}
By the identity in (\ref{eq:M210755p}), note that $(U^T U)^2 = (U^T U - f_1)^2 + 2 f_1  (U^T U - f_1) + f_1^2$, we obtain (\ref{numerator_cond_1}) with $\bar C_\delta = 2(4 q  \|\xi_1^2\|_{2q})^{2 q}$.

Let $U^* = A \vzeta$, where $\vzeta$ and $\vxi$ are i.i.d. Then $U^T U^* = \vxi^T B \vzeta$. By (\ref{eq:M21827}), $\RBR{\vxi^T B \vzeta}_{2q}^{2} \le (2q-1)^{2}\RBR{\xi_{1}}_{2q}^{4}f^{2}$, which similarly implies (\ref{numerator_cond_2}) with $\bar D_\delta = (4 q)^q \RBR{\xi_{1}}_{2q}^{2q} +(2q-1)^{2q}  \RBR{\xi_{1}}_{2q}^{4q}$ in view of the second identity in (\ref{eq:M210755p}). 
\end{proof}

\begin{proof}[Proof of Lemma \ref{lem:denominator}]
Let $\cB=\rbr{(i,j),1\le i\le j\le p}$ and $\vomg=\rbr{\omega_b}_{b\in\cB}\in\mathbb{R}^{p(p+1)/2}$, where $\omega_b=\xi_i\xi_j-\ind{i=j}$ for $b=(i,j)$, that is
\begin{align*}
\vomg=\rbr{\varrho_1,\xi_1\xi_2,\ldots,\xi_1\xi_p,\varrho_2,\xi_2\xi_3,\ldots,\varrho_p}^T, \mbox{ where }  \varrho_k=\xi_{k}^{2}-1.
\end{align*}
Let $V_{W}$ be the covariance matrix of $\vomg$. Then $V_{W}=\mathrm{diag}\rbr{\rBr{v_{b,b}}_{b\in\cB}}$, where for $b=\rbr{i,l}$, $v_{b,b} = \Var\rbr{\xi_i^2}$ if $l=i$ and $v_{b,b} = 1$ if $l \not= i$. Also define $G=\rbr{g_{a,b}}_{a\in \cI, b\in\cB} \in \mathbb{R}^{p^2\times\rbR{p(p+1)/2}}$, where for $a=\rbr{j,k}$, $b=\rbr{i,l}$,
\begin{align*}
g_{a,b}=\begin{cases}
a_{ji}a_{ki},&\mbox{ if }l=i;\\
a_{ji}a_{kl}+a_{jl}a_{ki},&\mbox{ if }i<l.
\end{cases}
\end{align*}
Note that $X_jX_k=\vg_a^T\vomg,$ where $\vg_a^T$ is the $a$'th row of $G$. Then $\vW=G\vomg$ and
\begin{align*}
\E\rbr{\vW\vW^T}=\rbr{\gamma_{a,a'}}_{a,a'\in\cI},
\end{align*}
where for $a=(j,k)$, $a'=(m,q)$,
\begin{eqnarray}\label{eq:M220636p}
\gamma_{a,a'} &=& \Cov\rbr{X_jX_k,X_mX_q}=\vg_a^TV_{W}\vg_{a'} \cr
 &=& \nu\sum_{i}a_{ji}a_{ki}a_{mi}a_{qi}+\sum_{i< l}\rbr{a_{ji}a_{kl}+a_{jl}a_{ki}}\rbr{a_{qi}a_{ml}+a_{mi}a_{ql}} \cr
 &=& (\nu-2)\sum_{i}a_{ji}a_{ki}a_{mi}a_{qi}+\sigma_{jm}\sigma_{kq}+\sigma_{jq}\sigma_{km}.
\end{eqnarray}
Let $B=A^TA=\rbr{b_{il}}_{i,l}$,  $L_0=2f^{4}+2f_{4}^{4}$, $L_1=\sum_{il}b_{il}^4$ and $L_2=\sum_i\rbr{\sum_l b_{il}^2}^2$. By (\ref{eq:M220636p}), 
\begin{align*}
f_\vW^2&=\sum_{a,a'\in\cI}\gamma_{a,a'}^2\\
&=\sum_{1\le j,k,m,q\le p}\big[{(\nu-2)\sum_{i}a_{ji}a_{ki}a_{mi}a_{qi}+\sigma_{jm}\sigma_{kq}+\sigma_{jq}\sigma_{km}}\big]^2\\
&= L_1(\nu-2)^2 + 4 L_2 (\nu-2) + L_0.
\end{align*}
Clearly $f_\vW^2\ge 2f^4$ if $\nu\ge 2$. Note that $4L_1-8L_2+L_0\ge 0$. Since $L_1\le L_2$, $L_0\ge 8L_2-4L_1\ge 4L_1$. If $0<\nu<2$, then the quantity
\[
f_\vW^2-\frac{L_0\nu^2}{4}=\rbr{L_1-\frac{L_0}{4}}\nu^2+4(L_2-L_1)\nu+L_0+4L_1-8L_2
\]
is larger than the minimum of its value at $\nu=0$ and $\nu=2$, which are both nonnegative. Therefore, $f_\vW^{2}\ge \nu^{2} f^{4}/2$ for any $\nu \in (0, 2)$. 
\end{proof}

\medskip


\bibliographystyle{plainnat}
\bibliography{l2asymp}

\end{document}